\documentclass[11pt]{amsartnew}
\usepackage{graphicx}
\usepackage{amssymb}
\usepackage{amsmath}
\usepackage{mathrsfs}
\usepackage{subfigure}
\usepackage{verbatim}
\usepackage{graphicx}
\usepackage{epsfig}
\usepackage{longtable}
\usepackage{epsfig,indentfirst,latexsym,bm,amsmath,amssymb,amsfonts,lscape,rawfonts,eufrak}
\usepackage{color}
\usepackage{algorithm,algorithmic}
\usepackage{cite}

\usepackage{setspace}

\newtheorem{theorem}{Theorem}[section]

\newtheorem{corollary}{Corollary}[section]

\newtheorem{example}{Example}[section]

\newtheorem{remark}{Remark}[section]

\numberwithin{equation}{section}

\def\R{{{\mathbb R}}}

\def\Z{{{\mathbb Z}}}

\def\bW{\bm{W}}

\newcommand{\cM}{{\mathcal M}}
\newcommand{\cP}{{\mathcal P}}
\newcommand{\cS}{{\mathcal S}}
\newcommand{\cT}{{\mathcal T}}

\newcommand{\cL}{{\mathcal L}}

\newcommand{\bnu}{\bm{\nu}}

\newcommand{\bu}{\bm{u}}

\date{}

\begin{document}

\title{Sparse Representation on Graphs by Tight Wavelet Frames and Applications}

\author{Bin Dong}
\thanks{Bin Dong is with the Beijing International Center for Mathematical Research(BICMR), Peking University,
Beijing, China, 100871.\\ \textit{Email:} dongbin@math.pku.edu.cn \\ \textit{Phone:} (8610)62744091.}
\thanks{This work was supported in part by NSF DMS-1418772, and by the Thousand Talents Plan of China.}


\begin{abstract}
In this paper, we introduce a new (constructive) characterization of tight wavelet frames on non-flat domains in both continuum setting, i.e. on manifolds, and discrete setting, i.e. on graphs; discuss how fast tight wavelet frame transforms can be computed and how they can be effectively used to process graph data. We start with defining the quasi-affine systems on a given manifold $\cM$. The quasi-affine system is formed by generalized dilations and shifts of a finite collection of wavelet functions $\Psi:=\{\psi_j: 1\le j\le r\}\subset L_2(\R)$. We further require that $\psi_j$ is generated by some refinable function $\phi$ with mask $a_j$. We present the condition needed for the masks $\{a_j: 0\le j\le r\}$, as well as regularity conditions needed for $\phi$ and $\psi_j$, so that the associated quasi-affine system generated by $\Psi$ is a tight frame for $L_2(\cM)$. The condition needed for the masks is a simple set of algebraic equations which are not only easy to verify for a given set of masks $\{a_j\}$, but also make the construction of $\{a_j\}$ entirely painless. Then, we discuss how the transition from the continuum (manifolds) to the discrete setting (graphs) can be naturally done. In order for the proposed discrete tight wavelet frame transforms to be useful in applications, we show how the transforms can be computed efficiently and accurately by proposing the fast tight wavelet frame transforms for graph data (WFTG). Finally, we consider two specific applications of the proposed WFTG: graph data denoising and semi-supervised clustering. Utilizing the sparse representation provided by the WFTG, we propose $\ell_1$-norm based optimization models on graphs for denoising and semi-supervised clustering. On one hand, our numerical results show significant advantage of the WFTG over the spectral graph wavelet transform (SGWT) by \cite{hammond2011wavelets} for both applications. On the other hand, numerical experiments on two real data sets show that the proposed semi-supervised clustering model using the WFTG is overall competitive with the state-of-the-art methods developed in the literature of high-dimensional data classification, and is superior to some of these methods.
\end{abstract}

\maketitle

\section{Introduction}

In recent years, we are experiencing rapid advances in information and computer technology, which contribute greatly to the exponential growth of data. To properly handle, process and analyze such huge and often unstructured data sets, sophisticated mathematical tools and efficient computing methods need to be developed. Such huge data sets, commonly referred to as ``big data'', are generally modelled as huge graphs living in very high dimensional spaces. Graphs are commonly understood as a certain discretization or a random sample from some smooth Riemannian manifold \cite{coifman2006diffusionmaps,belkin2005towards,hein2005graphs,gine2006empirical,singer2006graph}. To understand and analyze graphs and data on graphs (shall be called graph data), the graph Laplacian is widely used to reveal the geometric properties of the graph and plays an important role in many applications such as graph clustering.

In signal and image processing, many methods are transform based. Sparsity of the signal/image to be recovered under a certain transform is the key to the success of many existing algorithms. One of the successful examples is the wavelet frame transform, especially the tight wavelet frame transform \cite{chai2007deconvolution,chan2003wavelet,chan2004tight,cai2008restorationchopnod,CCSS,cai2008framelet,cai2008simultaneous,cai2008frameletJCM, chan2007frameletvideo,cai2009split,CDOS2011,DJS2013,CDS2014}. The power of tight wavelet frames lies in their ability to sparsely approximate piecewise smooth functions and the existence of fast decomposition and reconstruction algorithms. Recently, geometric properties of tight wavelet frames were discovered by connecting them to differential operators under variational and PDE frameworks \cite{CDOS2011,DJS2013,CDS2014}.

The success of wavelet frames for data defined on flat domains motivates much research on generalizing wavelets and wavelet frames to curved, irregular and unstructured domains. In this paper, we introduce a (constructive) characterization of tight wavelet frames on non-flat domains in both continuum (on manifolds) and discrete (on graphs) setting, discuss how fast tight wavelet frame transforms can be computed and how they can be effectively used to process and analyze graph data. The basic idea is to understand eigenfunctions of Laplace-Beltrami operator (graph Laplacian in discrete setting) as Fourier basis on manifolds (graphs in discrete setting) and the associated eigenvalues as frequency components. This idea was used earlier by \cite{hammond2011wavelets} in the discrete setting. In this paper, we further observe that Quasi-affine systems generated by dilations and shifts of wavelet functions can be defined on manifolds. When the elements in the quasi-affine system are generated from a refinable function, the transition from continuum (manifolds) to discrete (graphs) setting can be done very naturally. More importantly, such consideration makes the construction of various types of tight wavelet frames on manifolds/graphs totally painless, and it ensures the existence of fast decomposition and reconstruction algorithms which is crucial to many applications.

Given a compact and connected Riemannian manifold $(\cM, g)$, denote $L_2(\cM)$ the space of square integrable functions on $\cM$. We start with defining the quasi-affine system on $\cM$. The quasi-affine system is formed by generalized dilations and shifts of a finite collection of wavelet functions $\Psi:=\{\psi_j: 1\le j\le r\}\subset L_2(\R)$. We further restrict our consideration of $\Psi$ to those that are generated by a set of masks $\{a_j:\ 0\le j\le r\}\subset\ell_2(\Z)$. Then, we present the condition needed for the masks $\{a_j: 0\le j\le r\}$ (i.e. equation \eqref{UEP:1}) so that the associated quasi-affine system generated by $\Psi$ is a tight frame for $L_2(\cM)$ (Theorem \ref{Theorem:TightFrame:Manifold}). The condition on the masks is a simple set of algebraic equations which are not only easy to verify for a given set of masks $\{a_j\}$, but also make the construction of $\{a_j\}$ painless. In particular, we show that under suitable conditions, the quasi-affine system on $\cM$ generated by any set of the framelets constructed from the unitary extension principle \cite{ron1997affine} on $\mathbb{R}$ is a tight frame for $L_2(\cM)$ (Corollary \ref{Corollary:UEP:GraphFrame}). In addition, many masks constructed in \cite{DJS2013} satisfy the conditions \eqref{UEP:1} in Theorem \ref{Theorem:TightFrame:Manifold} as well, although they may not satisfy the unitary extension principle. Therefore, Theorem \ref{Theorem:TightFrame:Manifold} is not only a rather generic characterization of tight wavelet frames for $L_2(\cM)$, it also provides a simple way of verifying and constructing various types of tight wavelet frame systems on $L_2(\cM)$.

Thanks to the aforementioned special consideration on $\Psi$, i.e. associating $\Psi$ with a set of masks $\{a_j\}$, we discuss how the transition from the continuum (manifolds) to the discrete setting (graphs) can be naturally done. We show that inner products with wavelet frame functions on manifolds can be approximated in the discrete setting by ``filtering" with the associated masks on graphs. This leads to multi-level discrete tight wavelet frame transforms (decomposition and reconstruction) on graphs, which can be interpreted as an iterative ``convolution'' of graph data with properly dilated wavelet frame masks (or filters). In order for the proposed discrete tight wavelet frame transforms to be useful in applications, we show how the transforms can be computed efficiently and accurately. Since most masks we use are trigonometric polynomials, they can be accurately approximated by low-degree Chebyshev polynomials, which enables us to compute the decomposition and reconstruction transforms very efficiently and accurately without computing the eigenvalues and eigenvectors of the graph Laplacian. As a result, the proposed fast tight wavelet frame transform on graphs (WFTG) have a similar computational cost as the traditional fast wavelet frame transforms for images with negligible reconstruction error if the underlying graph Laplacian is a relatively sparse matrix. More importantly, numerical simulations show that the WFTG maps graph data to a set of \textit{sparse} coefficients, which indicates that the proposed tight wavelet frames indeed provide sparse representation on graphs.

Finally, we consider two specific applications of the proposed WFTG: graph data denoising and semi-supervised clustering. Utilizing the sparse representation provided by WFTG, we propose $\ell_1$-norm based optimization models on graphs for denoising and semi-supervised clustering problems. These models are motivated from models used in image restoration and image segmentation. We compare the proposed WFTG and the spectral graph wavelet transform (SGWT) of \cite{hammond2011wavelets} using the proposed models for denoising and semi-supervised clustering. Our numerical results show significant advantage of the WFTG over the SGWT for both applications on synthetic graph data. Furthermore, we compare the proposed semi-supervised clustering model using the WFTG with some of the state-of-the-art graph clustering models using two real data sets. Numerical studies show that our approach is overall competitive with these methods and is superior to some of the methods. We also note that, now that we have WFTG as an effective sparsifying transform for graph data, modeling on graphs can be easily generalized from modeling in image processing and analysis, whenever the same application is still relevant on graphs.

\subsection{Brief Literature Review}

Wavelets and their generalizations are well studied in the past thirty years \cite{Dau,chui1992introduction,mallat2008wavelet,Dong2010IASNotes}. Their success in various applications are mostly due to the sparse representation they provide for piecewise smooth functions. In the past few decades, there has been much endeavour in the community to generalize (bio)orthogonal wavelets and wavelet frames to functions or data defined on non-flat domains, such as spheres, surfaces, graphs, etc.

There has been a relatively rich literature of wavelets and wavelet frames on 2-dimensional surfaces. Wavelets on sphere were first introduced in \cite{SS} using the lifting scheme \cite{Sweldens}, and later in \cite{antoine1998wavelets,antoine1999wavelets} via a group-theoretical approach. Biorthogonal wavelets with high symmetry for surface multiresolution processing have been constructed in \cite{bertram2004biorthogonal,bertram2004generalized,jiang2011biorthogonal,jiang2011biorthogonal6,wang2006efficient,wang2007sqrt3}. Loop's scheme-based biorthogonal wavelets have been considered in \cite{khodakovsky2000progressive} with the biorthogonal dual wavelets constructed in \cite{han2005wavelets}.
Redundant representations on surfaces were introduced by \cite{jiang2011highly}, where 6-fold symmetric bi-frames with 4 framelets (frame generators) for triangulated surfaces were introduced. More recently in \cite{dong2015surf}, tight wavelet frames on triangulated and quad surfaces were constructed and applications in surface denoising were considered.

In recent years, there has been much interest in constructing wavelet-like representation of graph data. Successful examples include the
wavelets on unweighted graphs by \cite{crovella2003graph}, the multiscale scheme on graphs based on lifting by \cite{jansen2009multiscale}, the Haar wavelet transform for rooted binary trees \cite{murtagh2007haar} and its generalization treelets \cite{lee2008treelets}, the diffusion wavelets \cite{coifman2006diffusion} and diffusion polynomial frames \cite{maggioni2008diffusion}, the wavelets on compact differentiable manifolds \cite{geller2009continuous}, the spectral graph wavelet transform by \cite{hammond2011wavelets,gavish2010multiscale,leonardi2013tight}, Haar transform for coherent matrices \cite{gavish2012sampling}, and orthogonal polynomial systems for weighted trees \cite{chui2014representation}.

\subsection{Related Work and Contributions}

Coifman and Maggioni \cite{coifman2006diffusion} introduced diffusion wavelets on smooth manifolds as well as graphs. Their construction is based on repeated applications of a diffusion operator. The major difference between their work and ours is that the diffusion wavelets are orthonormal. Although an orthonormal system is desirable for some applications such as compression, redundant systems such as tight frames are more robust to errors than (bio)orthogonal systems. In many applications of transformation based data processing, most processing is done in the transform domain such as the widely used thresholding. Errors are inevitably introduced no matter how careful the thresholding operator is designed, especially at the presence of noise. However, if a redundant system is used, there is a good chance that these errors are canceled out after transforming back to the physical domain.

Redundant systems were considered by Maggioni and Mhaskar \cite{maggioni2008diffusion} where they developed a theory of diffusion polynomial frames that is related to our framework. However, they did not provide any algorithm for efficient computation of the decomposition and reconstruction transforms which are crucial to many applications. Geller and Mayeli \cite{geller2009continuous} studied a construction for wavelets on compact differentiable manifolds. In particular, their scaling is defined using a pseudodifferential operator $t\Delta e^{−t\Delta}$, where $\Delta$ is the Laplace-Beltrami operator on the given manifold and $t$ is the scaling parameter. Wavelets are obtained by applying the pseudodifferential operator to a delta impulse. In our framework, the scaling is defined by $\phi(t\Delta)$ where $\phi\in L_2(\R)$ is a refinable function. The wavelets are obtained by applying $\psi_j(t\Delta)$, with $\psi_j\in L_2(\R)$ and $1\le j\le r$, to delta impulses at each points of the given manifold. Moreover, we do not need to assume the manifolds are smooth for our approach.

The work by Hammond, Vandergheynst and Gribonval \cite{hammond2011wavelets} is most related to the present work. In \cite{hammond2011wavelets}, scaling is defined by $\phi(t\Delta)$ with $\phi\in L_2(\R)$ and wavelets are obtained by applying $\psi(t\Delta)$, with $\psi\in L_2(\R)$, to delta impulses at each vertex of the given graph. Fast decomposition and reconstruction algorithms based on Chebyshev polynomial approximation were proposed as well, which were referred to as the spectral graph wavelet transform (SGWT). The present work is different from \cite{hammond2011wavelets} in the following ways:
\begin{enumerate}
\item We consider characterization and construction of tight wavelet frames on both manifolds and graphs, while only graphs were considered by \cite{hammond2011wavelets}.
\item We consider wavelet frame functions $\psi_j$ that are generated by a refinable function $\phi$ via the associated mask $a_j$, which is an entirely new approach. Such consideration grants a natural transition from continuum (manifolds) to discrete (graphs) settings. More importantly, due to the association of $\psi_j$ with mask $a_j$, the proposed fast transform, i.e. WFTG, is more efficient than the SGWT in the discrete setting (graphs), although the essential tool used by both approaches is the Chebyshev polynomial approximation. This is because for WFTG, the function needs to be approximated is $\widehat a_j$ (instead of $\widehat\psi_j$ itself as in \cite{hammond2011wavelets}), which is a trigonometric polynomial. Therefore, the WFTG can be computed using low-degree Chebyshev polynomial approximations, instead of high-degree Chebyshev polynomial approximations needed for the SGWT. Efficient and accurate computation of the forward and inverse transforms are rather important in many applications.
\item Theorem \ref{Theorem:TightFrame:Manifold} (continuum) and Theorem \ref{Theorem:TightFrame:Graph} (discrete) provide a constructive characterization of tight wavelet frames on manifolds and graphs. Systematic construction of tight wavelet frames is totally painless using our approach.
\item The tight wavelet frame systems constructed by our approach are potentially more effective in applications than those constructed by \cite{hammond2011wavelets}, as indicated by the comparisons of the WFTG and SGWT given in Section \ref{Sec:Applications} on denoising and semi-supervised clustering problems.
\end{enumerate}

\subsection*{Contributions}
The contribution of this paper is fourfold.
\begin{enumerate}
\item \textit{Painless construction of tight wavelet frames on manifolds and graphs.} The construction of tight wavelet frames is essentially reduced to the construction of a set of masks $\{a_j: 0\le j\le r\}$ such that a set of algebraic equations of $\widehat a_j$ is satisfied, i.e. equation \eqref{UEP:1}. In particular, any masks constructed from the unitary extension principle, such as those in \cite{ron1997affine,ron1998compactly,grochenig1998tight,chui2000compactly,petukhov2001explicit,selesnick2001smooth,petukhov2003symmetric,Daubechies2003,DSpseudospline,han2009dual} and many masks constructed in \cite{DJS2013} satisfy the condition \eqref{UEP:1}.
\item \textit{Fast transforms available, i.e. the fast tight wavelet frame transform on graphs (WFTG).} The key ingredient of the proposed WFTG is the Chebyshev polynomial approximation of $\widehat a_j$. For a given finitely supported mask $a_j$, its Fourier series $\widehat a_j$ is a trigonometric polynomial and hence is analytic. Therefore, it can be accurately approximated by low-degree Chebyshev polynomials (see e.g. \cite{mason2010chebyshev}).
\item \textit{The WFTG is more effective than the SGWT in solving denoising and semi-supervised clustering problems.} Numerical examples in Section \ref{Sec:Applications} shows that using the proposed denoising and semi-supervised clustering model, better denoising and clustering results are obtained when the WFTG is used.
\item \textit{An effective semi-supervised clustering model using the WFTG is proposed and tested on two real data sets.} High-dimensional data clustering is one of the most important problems in machine learning and big data analysis. Numerical experiments in Section \ref{SubS:Clustering} showed that the proposed clustering model using the WFTG is overall competitive with the state-of-the-art methods and is superior to some of these methods.
\end{enumerate}

\section{Tight Wavelet Frames on Manifold $\{\cM, g\}$}\label{Sec:Continuous:Transforms}

We start this section with a brief review of some results in eigenvalue problems of Laplace-Beltrami operator on Riemannian manifolds. For a comprehensive review of the subject, one can refer to the book \cite{chavel1984eigenvalues}. Given a Riemannian manifold $(\cM, g)$, we will define a quasi-affine system generated by finitely many functions, and discuss the conditions needed for a quasi-affine system to form a tight frame for $L_2(\cM)$. We end this section by showing some examples.

\subsection{Spectrum and Eigenfunctions of Laplace-Beltrami Operator} Let $\{\cM, g\}$ be a compact, connected Riemannian manifold with smooth boundary $S$. Let $m\ge2$ be the dimension of $\cM$. Let $\Delta$ be the Laplace-Beltrami operator on $\cM$ with respect to the metric $g$. Let $\{\lambda_p: p=0,1,\ldots\}$ and $\{u_p: p=0,1,\ldots\}$ be the eigenvalues and eigenfunctions of the following eigenvalue problem
\begin{equation}\label{E:Eigenvalue}
\Delta u+\lambda u=0,
\end{equation}
with Dirichlet boundary condition $u_{|S}=0$. As convention, $0<\lambda_0\le\lambda_1\le\lambda_2\le\cdots$. The set of eigenfunctions form an orthonormal basis of $L_2(\cM)$, i.e. $$\langle u_p, u_{p'}\rangle_{L_2(\cM)}=\int_{\cM}u_{p}(x)u_{p'}^*(x)dx=\delta_{p,p'},$$ and $\{u_p\}$ is complete in $L_2(\cM)$. Here, $f^*$ denotes the complex conjugate of $f$. Given $f\in L_2(\cM)$, we define $$\widehat f [p]=\langle f, u_p \rangle_{L_2(\cM)},$$ with $\widehat f\in\ell_2(\Z^+)$ and $\Z^+=\{0,1,2,\ldots\}$. Then, we have the following identity
$$\langle f, g \rangle_{L_2(\cM)}=\langle\widehat f,\widehat g\rangle_{\ell_2(\Z^+)}\quad\mbox{for } f,g\in L_2(\cM),$$ where $\langle\widehat f,\widehat g\rangle_{\ell_2(\Z^+)}:=\sum_{p=0}^\infty \widehat f[p]\widehat g^*[p]$. In particular, we have the following Parseval's identity $$\|f\|_{L_2(\cM)}^2=\|\widehat f\|_{\ell_2(\Z^+)}^2.$$ For convenience, we shall drop the subscript in $\langle \cdot, \cdot \rangle_{L_2(\cM)}$ and $\langle\cdot,\cdot\rangle_{\ell_2(\Z^+)}$ whenever there is no confusion.

Finally, we recall the following two results from eigenvalue problems of Riemannian geometry. First, Weyl's asymptotic formula \cite{weyl1912asymptotische,chavel1984eigenvalues} gives us the growth rate of $\{\lambda_p\}$:
\begin{equation}\label{E:Weyl}
\lambda_p \asymp p^{\frac2{m}},
\end{equation}
with $m$ the dimension of the manifold. Here, $a_n\asymp b_n$ means $0 < c = \liminf_n |a_n/b_n| \leq \limsup_n |a_n/b_n| = C < \infty$. The second result is the uniform bound of the eigenfunctions \cite{grieser2002uniform}:
\begin{equation}\label{E:Grieser}
\|u_p\|_{L_\infty(\cM)}\le C\lambda_p^{\frac{m-1}{4}}.
\end{equation}

\subsection{Tight Wavelet Frames for $L_2(\cM)$}
We start with defining a system on $\cM$ that mimics the traditional quasi-affine system \cite{ron1997affine,Dong2010IASNotes} generated by finitely many functions defined on $\R$. With an abuse of terminology, we shall refer to such systems on $\cM$ as quasi-affine systems as well. The quasi-affine system on manifold $\cM$ is generated by dyadic dilations and continuum translations of finitely many functions defined on the real line $\R$. Given $f\in L_2(\R)$, define dilation and translation of $f$ as
\begin{equation}\label{D:Dilation:Translation:General}
f^\cM_{n,y}(x):=\sum_{p=0}^\infty \widehat{f}(2^{-n}\lambda_p)u_p^*(y)u_p(x), \quad\mbox{with } n\in\Z,\ x\in\cM,\ y\in\cM.
\end{equation}
Here, $\widehat f$ is the Fourier transform on the real line. The index $n$ of $f^\cM_{n,y}$ denotes dilation and $y$ translation. The dyadic dilation $2^{-n}$ used here is nonessential and one may change it to any $\gamma^{-n}$ with $\gamma>1$. As one can see that $f^\cM_{n,y}$ is defined in the spectral domain mimicking the dilation and translation of functions on Euclidean domains via Fourier transform, which was first used by \cite{hammond2011wavelets} in discrete setting, i.e. on graphs. Note that the dilation and translation is only formally defined in \eqref{D:Dilation:Translation:General}, since in general the infinite summation on the right hand side may diverge. To make sure $f^\cM_{n,y}\in L_2(\cM)$, proper decay condition on $\widehat f$ is needed. We will return to this with more rigorousness after we define the quasi-affine systems.

Define the quasi-affine system $X(\Psi)\subset L_2(\cM)$ generated by $\Psi:=\{\psi_j:\ 1\le j\le r\}\subset L_2(\R)$ as
\begin{equation}\label{D:AffineSystem}
X(\Psi):=\{\psi^\cM_{j,n,y}\in L_2(\cM):\ 1\le j\le r, n\in\Z, y\in\cM\},
\end{equation}
where $\psi^\cM_{j,n,y}\in L_2(\cM)$ is the dilation and translation of $\psi_j$ as defined in \eqref{D:Dilation:Translation:General}:
\begin{equation}\label{D:Dilation:Translation}
\psi^\cM_{j,n,y}(x):=\sum_{p=0}^\infty \widehat{\psi}_j(2^{-n}\lambda_p)u_p^*(y)u_p(x), \quad\mbox{with } n\in\Z,\ x\in\cM,\ y\in\cM.
\end{equation}
To ensure $\psi^\cM_{j,n,y}\in L_2(\cM)$, we need to impose a regularity condition on $\psi_j\in L_2(\R)$ or, in other words, a decay condition on $\widehat\psi_j$ as follows:
\begin{equation}\label{Assumption:Decay:Psi}
|\widehat\psi_j(\xi)|\le C\left(1+|\xi|\right)^{-s}\quad \mbox{with } s>\frac{2m-1}{4},
\end{equation}
for all $\xi\in\R$. Then, the two results in geometry \eqref{E:Weyl} and \eqref{E:Grieser} guarantee that $\psi^\cM_{j,n,y}\in L_2(\cM)$. Note that \eqref{D:Dilation:Translation} can be written equivalently in Fourier domain as
\begin{equation}\label{D:Dilation:Translation:Fourier}
\widehat{\psi^\cM_{j,n,y}}[p]:=\widehat{\psi}_j(2^{-n}\lambda_p)u_p^*(y), \quad\mbox{with } n\in\Z,\ p\in\Z^+,\ y\in\cM.
\end{equation}

The main objective of this section is to discuss conditions on $\psi_j\in L_2(\R)$, $1\le j\le r$, such that $X(\Psi)$ defined by \eqref{D:AffineSystem} is a tight frame for $L_2(\cM)$. In case $X(\Psi)$ is a tight frame for $L_2(\cM)$, we shall call $X(\Psi)$ a tight wavelet frame for $L_2(\cM)$ and the elements in $\Psi$ framelets.

\subsubsection{Characterization of Tight Wavelet Frames for $L_2(\cM)$}

We consider the set of wavelet functions $\Psi$ that are finite linear combinations of the shifts of a certain refinable (scaling) function $\phi\in L_2(\R)$, rather than a general set of wavelet functions. This makes our approach different from the existing characterization and construction of wavelet (frame) systems on manifolds/graphs such as \cite{coifman2006diffusion,geller2009continuous,hammond2011wavelets,gavish2010multiscale,leonardi2013tight}. Such consideration makes the construction of tight wavelet frames for $L_2(\cM)$ totally painless. In fact, as will be shown by Corollary \ref{Corollary:UEP:GraphFrame} that all tight wavelet frames constructed from the unitary extension principle of \cite{ron1997affine} can also generate tight wavelet frame systems for $L_2(\cM)$. Furthermore, such consideration ensures the existence of fast and accurate decomposition and reconstruction algorithms that can be easily implemented as well (see Section \ref{Sec:Discrete:WFTG}). These properties are rather crucial to many applications.

Let $\phi\in L_2(\R)$ be a compactly supported refinable function with finitely supported refinement mask $a\in\ell_0(\Z)$ satisfying the refinement equation $$\widehat\phi(2\xi)=\widehat a(\xi)\widehat\phi(\xi).$$ Here, $\widehat \phi$ is the Fourier transform of $\phi\in L_2(\R)$ and $\widehat a$ is the Fourier series of $a$, which is a trigonometric polynomial since $a$ is finitely supported. Same as the wavelet functions $\psi_j$, we assume the following decay property on $\widehat\phi$ so that the dilation and translation of $\phi$ is an element in $L_2(\cM)$:
\begin{equation}\label{Assumption:Decay:phi}
|\widehat\phi(\xi)|\le C\left(1+|\xi|\right)^{-s}\quad \mbox{with } s>\frac{2m-1}{4}.
\end{equation}
Given such refinable function $\phi\in L_2(\R)$ with mask $a\in\ell_0(\Z)$, we define a set of compactly supported functions $\Psi:=\{\psi_1,\ldots,
\psi_r\}\subset L_2(\R)$ by their associated masks $a_j\in\ell_0(\Z)$:
\begin{equation*}
\widehat\psi_j(2\xi):=\widehat a_j(\xi)\widehat\phi(\xi),\quad 1\le j\le r.
\end{equation*}
Letting $\psi_0:=\phi$ and $a_0:=a$, we can include the refinement equation for $\phi$ in the above equations, i.e.
\begin{equation}\label{D:MRA:System}
\widehat\psi_j(2\xi)=\widehat a_j(\xi)\widehat\phi(\xi)\quad\mbox{for } 0\le j\le r.
\end{equation}

Given $\psi^\cM_{j,n,y}\in L_2(\cM)$ defined by \eqref{D:Dilation:Translation}, the inner product between $f\in L_2(\cM)$ and $\psi^\cM_{j,n,y}$ satisfies
\begin{equation}\label{D:Coeff:Continuum}
\langle f, \psi^\cM_{j,n,y} \rangle=\langle \widehat f, \widehat{\psi^\cM_{j,n,y}} \rangle=\sum_{p=0}^\infty \widehat{f}[p]\widehat{\psi}_j^*(2^{-n}\lambda_p)u_p(y),
\end{equation}
where the second equality follows from \eqref{D:Dilation:Translation:Fourier}. Note that the decay conditions \eqref{Assumption:Decay:Psi} and \eqref{Assumption:Decay:phi} implies that $\langle f, \psi^\cM_{j,n,\cdot} \rangle \in L_2(\cM)$ for $0\le j\le r$. Then, \eqref{D:Coeff:Continuum} can be written equivalently in Fourier domain as
\begin{equation}\label{D:Coeff:Continuum:Fourier}
\widehat{\langle f, \psi^\cM_{j,n,\cdot} \rangle}[p]=\widehat{f}[p]\widehat{\psi}_j^*(2^{-n}\lambda_p).
\end{equation}
Now, we can define the following operator $\cP_{n,j}$ on $L_2(\cM)$ as
\begin{equation*}
\cP_{n,j}f:=\int_\cM \langle f, \psi^\cM_{j,n,y}\rangle\psi^{\cM}_{j,n,y}dy,\quad 0\le j\le r.
\end{equation*}
In particular,  $$\cP_{n}f:=\cP_{n,0}=\int_\cM \langle f, \phi^\cM_{n,y}\rangle\phi^{\cM}_{n,y}dy.$$ Now, we show that the operator $\cP_{n,j}$ maps $L_2(\cM)$ into itself. Note that \eqref{D:Dilation:Translation} implies that the Fourier coefficients of $\psi^{\cM*}_{j,n,y}(x)$ (the complex conjugate of $\psi^{\cM}_{j,n,y}(x)$) with respect to $y$ satisfy
$$\widehat{\psi^{\cM*}_{j,n,\cdot}(x)}[p]=\widehat{\psi}^*_j(2^{-n}\lambda_p)u^*_p(x).$$ Then, by the above identity and \eqref{D:Coeff:Continuum:Fourier}, we have
\begin{equation*}
\begin{split}
(\cP_{n,j}f)(x)&=\langle\langle f, \psi^\cM_{j,n,\cdot} \rangle, \psi^{\cM^*}_{j,n,\cdot}\rangle=\langle\widehat{\langle f, \psi^\cM_{j,n,\cdot} \rangle}, \widehat{\psi^{\cM^*}_{j,n,\cdot}}\rangle\cr
 &=\sum_{p=0}^\infty\widehat f[p]\left|\widehat{\psi}_j(2^{-n}\lambda_p)\right|^2u_p(x).
\end{split}
\end{equation*}
By the decay conditions \eqref{Assumption:Decay:Psi} and \eqref{Assumption:Decay:phi} on $\widehat{\psi}_{j}$ for $0\le j\le r$, it is obvious that $\cP_{n,j}f\in L_2(\cM)$ for every $f\in L_2(\cM)$. Furthermore, we have
\begin{equation}\label{E:Pnj:Fourier}
\widehat{\cP_{n,j}f}[p]=\widehat f[p]\left|\widehat{\psi}_j(2^{-n}\lambda_p)\right|^2,\qquad\mbox{for } 0\le j\le r.
\end{equation}

By definition, $\{\psi_j:\ 0\le j\le r\}$ is uniquely determined by the set of masks $\{a_j:\ 0\le j\le r\}$. Thus, the system $X(\Psi)$ is uniquely determined by masks $\{a_j:\ 0\le j\le r\}$. The following theorem tells us for what conditions on $\{a_j:\ 0\le j\le r\}$, the corresponding system $X(\Psi)$ defined by \eqref{D:AffineSystem} is a tight frame for $L_2(\cM)$. In fact, the decay conditions \eqref{Assumption:Decay:Psi} and \eqref{Assumption:Decay:phi} can be characterized by the associated masks as well \cite{Dau}. However, for simplicity and clarity, we will not convert \eqref{Assumption:Decay:Psi} and \eqref{Assumption:Decay:phi} to conditions on the masks.

\begin{theorem}\label{Theorem:TightFrame:Manifold}
Given the set of compactly supported functions $\{\psi_j:\ 0\le j\le r\}\subset L_2(\R)$ and the associated trigonometric polynomials (or masks) $\{\widehat a_j(\xi):\ 0\le j\le r\}$ satisfying \eqref{D:MRA:System}, assume that the decay conditions \eqref{Assumption:Decay:Psi} and \eqref{Assumption:Decay:phi} are satisfied,
\begin{equation}\label{Assumption:mask}
|\widehat a_0(\xi)-1|\le C|\xi|
\end{equation}
for $\xi$ near the origin, and
\begin{equation}\label{UEP:1}
\sum_{j=0}^r\left|\widehat{a}_j(\xi)\right|^2=1.
\end{equation}
Then, the system $X(\Psi)$ is a tight frame for $L_2(\cM)$, i.e.
\begin{equation*}
f=\sum_{j=1}^r\sum_{n\in\Z}\int_\cM\langle f, \psi^\cM_{j,n,y}\rangle\psi^\cM_{j,n,y}dy\qquad \mbox{for every } f\in L_2(\cM).
\end{equation*}
\end{theorem}

\begin{proof}
By \eqref{E:Pnj:Fourier}, we have
\begin{equation*}
\begin{split}
\widehat{\cP_{n}f}[p]=\widehat{\cP_{n,0}f}[p]&=\widehat{f}[p]\left|\widehat{\phi}(2^{-n}\lambda_p)\right|^2\cr
 (\mbox{by \eqref{UEP:1}})\qquad &=\sum_{j=0}^r\widehat{f}[p]\left|\widehat{a}_j(2^{-n}\lambda_p)\widehat{\phi}(2^{-n}\lambda_p)\right|^2\cr
 &=\sum_{j=0}^r\widehat{f}[p]\left|\widehat{\psi}_j(2^{-n+1}\lambda_p)\right|^2=\sum_{j=0}^r\widehat{\cP_{n-1,j}f}[p].
\end{split}
\end{equation*}
Thus, $$\cP_n f=\cP_{n-1}f+\sum_{j=1}^r\cP_{n-1,j}f$$ and hence $$\cP_{n_1} f=\cP_{n_2}f+\sum_{j=1}^r\sum_{n=n_2}^{n_1}\cP_{n,j}f.$$

We first show that $$\cP_{n_1} f\to f\quad\mbox{in }L_2(\cM)\ \mbox{as } n_1\to\infty.$$ By the assumption \eqref{Assumption:mask}, we have $\widehat\phi(\xi)\to 1$ as $\xi\to0$ (see e.g. \cite{cavaretta1991stationary,jiang1999existence}).
Then,
\begin{equation*}
\|\cP_{n_1} f-f\|_{L_2(\cM)}^2=\sum_{p=0}^\infty\left|\widehat f[p]|\widehat\phi(2^{-n_1}\lambda_p)|^2-\widehat f[p]\right|^2=\sum_{p=0}^\infty|\widehat f[p]|^2\left||\widehat\phi(2^{-n_1}\lambda_p)|^2-1\right|^2.
\end{equation*}
Since $\widehat\phi(2^{-n_1}\lambda_p)\to 1$ for each $p\ge0$ as $n_1\to\infty$, and $\widehat\phi$ is bounded, we have $\|\cP_{n_1} f-f\|_{L_2(\cM)}^2\to0$ as $n_1\to\infty$ (by the dominated convergence theorem).

Finally, we show that $$\cP_{n_2} f\to0 \quad\mbox{in }L_2(\cM)\ \mbox{as } n_2\to-\infty.$$ Indeed, since $\widehat\phi(\xi)\to 0$ as $\xi\to\infty$, we have $\widehat{\cP_{n_2}f}[p]=\widehat{f}[p]|\widehat{\phi}(2^{-{n_2}}\lambda_p)|^2\to 0$ as $n_2\to-\infty$, for $p\ge0$. Thus, the boundedness of $\widehat\phi$ leads to $\cP_{n_2} f\to0$ in $L_2(\cM)$ as $n_2\to-\infty$ (by the dominated convergence theorem). This concludes the proof of the theorem.
\end{proof}

\begin{remark}
\hspace*{5in}
\begin{enumerate}
\item  From Theorem \ref{Theorem:TightFrame:Manifold}, we can see that by considering $\Psi=\{\psi_j\}_j$ that is obtained from the identity \eqref{D:MRA:System}, we can simply start with a properly chosen refinable function $\phi$ and find a set of masks $\{a_j\}_j$ such that \eqref{UEP:1} is satisfied. The decay conditions can also be verified using the masks \cite{Dau}. This makes the construction of tight wavelet frames for $L_2(\cM)$ very simple. Such construction of $\Psi$ is known as the multiresolution analysis (MRA) \cite{mallat1989multiresolution,meyer1992wavelets,de1993construction,jia1994multiresolution} based construction in the literature of wavelets \cite{daubechies1988orthonormal,cohen1992biorthogonal,daubechies1988orthonormal} and wavelet frames \cite{ron1997affine,chui2002compactly,Daubechies2003,han2009dual} on flat domains.
\item We can see from the above proof that, if Neumann boundary condition is chosen for the eigenvalue problem \eqref{E:Eigenvalue} or the underlying manifold $\cM$ has no boundary at all, we will have $\widehat{\cP_{n_2}f}[0]\to \widehat f[0]$ and $\widehat{\cP_{n_2}f}[p]\to 0$ for $p\ge1$ as $n_2\to-\infty$, since $\lambda_0=0$. Therefore, the system $X(\Psi)$ obtained from Theorem \ref{Theorem:TightFrame:Manifold} is a tight wavelet frame for the subspace $\{f\in L_2(\cM): \langle f, u_0 \rangle=0\}$ instead of $L_2(\cM)$ itself.
\end{enumerate}
\end{remark}

By Theorem \ref{Theorem:TightFrame:Manifold}, it is immediate that all the compactly supported tight wavelet frames for $L_2(\R)$ constructed from the unitary extension principle (UEP) \cite{ron1997affine} can generate a tight wavelet frame for $L_2(\cM)$ provided that the decay conditions \eqref{Assumption:Decay:Psi} and \eqref{Assumption:Decay:phi} are satisfied. Recall that a set of masks $\{a_j\in\ell_0(\Z): j=0,1,\ldots,r\}$ is said to satisfy the UEP \cite{ron1997affine} if
\begin{equation}\label{UEP}
\sum_{j=0}^r\left|\widehat{a}_j(\xi)\right|^2=1\quad\mbox{and}\quad\sum_{j=0}^r\widehat{a}_j(\xi)\widehat{a}_j^*(\xi+\pi)=0.
\end{equation}
Interested reader should consult \cite{ron1997affine,Daubechies2003} for details.

\begin{corollary}\label{Corollary:UEP:GraphFrame}
Let $\{\psi_j: 0\le j\le r\}$ and the associated trigonometric polynomials (or masks) $\{\widehat a_j:\ 0\le j\le r\}$ be constructed from the UEP. Then $X(\Psi)$ defined by \eqref{D:AffineSystem} is a tight frame for $L_2(\cM)$ provided that the decay conditions \eqref{Assumption:Decay:Psi} and \eqref{Assumption:Decay:phi} are satisfied and $|\widehat a_0(\xi)-1|\le C|\xi|$ for $\xi$ near the origin.
\end{corollary}

\begin{remark}
By Theorem \ref{Theorem:TightFrame:Manifold}, the second condition of the UEP is not needed for $X(\Psi)$ defined by \eqref{D:AffineSystem} to form a tight frame for $L_2(\cM)$. This is because the system $X(\Psi)$ is entirely translation invariant. The second condition of the UEP is to make sure aliasing can be canceled, which is not needed for a translation invariant wavelet system (see e.g. \cite[Theorem 6.4]{fan2014duality}). Therefore, other than the masks constructed form the UEP, there are many more examples of $\{a_j\}_j$ whose corresponding function $\psi_j$ generate a tight wavelet frame for $L_2(\cM)$. For example, many masks constructed by \cite{DJS2013} for discretization of nonlinear PDEs only satisfy \eqref{UEP:1}. If their corresponding refinable and wavelet functions satisfy the decay conditions, they generate tight frame systems for $L_2(\cM)$ by Theorem \ref{Theorem:TightFrame:Manifold}.
\end{remark}

\subsubsection{Examples of Tight Wavelet Frames on $\cM$}

Theorem \ref{Theorem:TightFrame:Manifold} tell us that to construct tight wavelet frames for $L_2(\cM)$, we can simply start with finding a finite set of filters such that \eqref{UEP:1} is satisfied. In particular, Corollary \ref{Corollary:UEP:GraphFrame} tells us that, under suitable conditions, tight wavelet frames constructed from the UEP can generate tight wavelet frames for $L_2(\cM)$. This leads us to a huge collection of tight wavelet frames, because in the literature of wavelet frames, numerous sets of masks with various different properties are constructed based on the UEP \cite{ron1997affine,ron1998compactly,grochenig1998tight,chui2000compactly,petukhov2001explicit,selesnick2001smooth,petukhov2003symmetric,Daubechies2003,DSpseudospline,han2009dual}. Here, we shall focus on the sets of masks constructed from B-splines in \cite{ron1997affine}.

Consider a B-spline $\phi$ of order $r$ with $r\ge1$:
\begin{equation}\label{D:Bsplines}
\widehat {\phi}(\xi)=e^{-ij\frac{\xi}{2}}{\bigg
(}\frac{\sin(\xi/2)}{\xi/2}{\bigg )}^{r},
\end{equation}
with $\tau=0$ when $r$ is even and $\tau=1$ when $r$ is odd. The corresponding refinement mask $\widehat a_0$ is given as $$\widehat
a_0(\xi)=e^{-i\tau\frac{\xi}{2}}\cos^{r}(\xi/2).$$ We define $r$ wavelet masks as
\begin{equation}\label{Masks:BSplineTightFramelets}
\widehat a_j(\xi)=-i^j e^{-i\tau\frac{\xi}{2}}\sqrt{{r \choose j}}\sin^j(\xi/2)\cos^{r-j}(\xi/2),\quad 1\le j\le r.
\end{equation}
It is easy to check that \eqref{UEP:1} is satisfied:
$$\sum_{j=0}^{r}|\widehat a_j(\xi)|^2=(\cos^2(\xi/2)+\sin^2(\xi/2))^{r}=1.$$
Therefore, by Theorem \ref{Theorem:TightFrame:Manifold}, the system $X(\Psi)$ generated by the $r$ framelets defined by
\begin{equation}\label{D:WaveletFrame:BSpline}
\widehat\psi_j=-i^j e^{-i\tau\frac{\xi}{2}}\sqrt{{r \choose
j}}\frac{\cos^{r-j}(\xi/4)\sin^{r+j}(\xi/4)}{(\xi/4)^{r}},\quad 1\le j\le r,
\end{equation}
forms a tight frame for $L_2(\cM)$ for all $r\ge \frac{2m-1}{4}$ with $m$ the dimension of the manifold $\cM$.

%
%

\begin{remark}
When the dimension of $\cM$ is 2, i.e. $m=2$, the framelets $\Psi$ given by \eqref{D:WaveletFrame:BSpline} for all $r\ge1$ can generate a tight wavelet frame for $L_2(\cM)$. However, when $m\ge3$, according to the assumption \eqref{Assumption:Decay:Psi}, only the wavelet frame systems generated by higher order B-splines can form tight frames. Of course, this assumption may be weakened so that the systems generated by low order B-splines also form a tight frame for $L_2(\cM)$. In the next section, we shall show that in the discrete setting, the set of masks \eqref{Masks:BSplineTightFramelets} for any $r\ge1$ always generates a discrete tight frame on graphs. More generally, any set of masks $\{a_j:\ 0\le j\le r\}$ satisfying \eqref{UEP:1} can generate a discrete tight wavelet frame on graphs (see Theorem \ref{Theorem:TightFrame:Graph}).
\end{remark}

\section{Discrete Tight Wavelet Frame Transforms}\label{Sec:Discrete:WFTG}

\subsection{Motivations}

Let $h_{j,n}(y):=\langle f, \psi^\cM_{j,n,y} \rangle$. Observe that
\begin{equation}\label{E:Transition}
\begin{split}
\widehat{h}_{j,n-1}[p]&= \widehat{f}[p]\widehat{\psi}_j^*(2^{-n+1}\lambda_p)\\
 & = \widehat{f}[p]\widehat{a}_j^*(2^{-n}\lambda_p)\widehat{\phi}^*(2^{-n}\lambda_p)\\
 &= \widehat{a}_j^*(2^{-n}\lambda_p)\widehat{h}_{0,n}[p].
\end{split}
\end{equation}
This shows that $h_{j,n-1}$, i.e. the continuous wavelet frame coefficients at level $n-1$ and band $j$, can be obtained from $h_{0,n}$, i.e. the low frequency coefficients at level $n$, by ``convolving'' $h_{0,n}$ with the masks $a_j$. As further discussed below, we can understand discrete function $f_G$ on a graph (which is a certain discretization of $\cM$) as a sampling of the underlying function $f$ by sampling $h_{0,n}(y)$ at certain scale $n>0$. Then, the discrete wavelet frame transforms can be simply defined as ``convolutions'' of $f_G$ with filters $\{a_j\}$.

\subsection{Discrete Tight Wavelet Frame Transforms on Graphs}
We denote a graph as $G:=\{E, V, w\}$, where $V:=\{v_k\in\cM: k=1,\ldots,K\}$ is a discretization of a given manifold $\cM$, $E\subset V\times V$ is an edge set, and $w: E \mapsto\R^+$ denotes a weight function. In this paper, we choose the following commonly used weight function $$w(v_k,v_{k'}):=e^{-\|v_k-v_{k'}\|_2^2/\sigma},\quad \sigma>0.$$ Let $A:=(a_{k,k'})$ be the adjacency matrix
\begin{equation*}
a_{k,k'}:=
\begin{cases}
w(v_{k},v_{k'}) & \mbox{if } v_{k} \mbox{ and } v_{k'} \mbox{ are connected by an edge in } E\\
0 & \mbox{otherwise,}
\end{cases}
\end{equation*}
and $D:=\mbox{diag}\{d[1], d[2], \ldots, d[K]\}$ where $d[k]$ is the degree of node $v_k$ defined by $d[k]:=\sum_{k'}a_{k,k'}$. Let $\cL_{0}$ be the (unnormalized) graph Laplacian, which takes the following form
\begin{equation*}
\cL_{0}:=D-A.
\end{equation*}
In the literature, normalized graph Laplacians were also considered $$\cL_{1}:=I-D^{-1}A\quad\mbox{or}\quad\cL_{2}:=I-D^{-1/2}AD^{1/2}.$$ Here, we shall use $\cL$ to denote one of the above three graph Laplacians. In our numerical computations, we will use the unnormalized graph Laplacian. The consistency of the graph Laplacian to the Laplace-Beltrami operator was studied in \cite{belkin2005towards,hein2005graphs,gine2006empirical}.

Denote $\{(\lambda_k, u_k)\}_{k=0}^{K-1}$ the set of pairs of eigenvalues and eigenfunctions of $\cL$. Assuming the graph is connected, then we have $0=\lambda_0<\lambda_1\le\lambda_2\le\cdots\le\lambda_{K-1}$. The eigenfunctions form an orthonormal basis for all functions on the graph: $$\langle u_k, u_{k'}\rangle:=\sum_{n=1}^K u_k[n]u_{k'}[n]=\delta_{k,k'}.$$ Let $f_{G}: V\mapsto \R$ be a function on the graph $G$. Then its Fourier transform is given by $$\widehat{f_G}[k]:=\sum_{n=1}^K {f}_G[n]u_k[n].$$

\subsection*{Transition from continuum to discrete} Suppose $G$ is a certain discretization of $\cM$ with $V\subset\cM$. Given a function $f_G$ on graph $G$, we assume that $f_G$ is sampled from the underlying function $f: \cM\mapsto\R$ by
$$f_G[k]:=\langle f,\phi_{N,v_k} \rangle,\quad v_k\in V,$$ with the dilation scale $N$ being the smallest integer such that $$\lambda_{\max}:=\lambda_{K-1}\le 2^{N}\pi.$$ Note that the scale $N$ is selected such that $2^{-N}\lambda_k\in[0,\pi]$ for $0\le k\le K-1$.

Based on our earlier observation \eqref{E:Transition}, we define the \textit{discrete $L$-level tight wavelet frame decomposition} as $$\bm{W}f_G:=\left\{W_{j,l}{f_G}: (j,l)\in\mathbb{B}\right\}$$ with
\begin{equation}\label{D:WT:IndexSet}
\mathbb{B}:=\{(1,1),(2,1),\ldots,(r,1),(1,2),\ldots,(r,L)\}\cup\{(0,L)\}
\end{equation}
and
\begin{equation}\label{D:TightFrame:Discrete}
\widehat{W_{j,l}{f_G}}[k]:=
\begin{cases}
\widehat{a}_{j}^*(2^{-N}\lambda_k)\widehat{{f_G}}[k] & l=1,\\
\widehat{a}_{j}^*(2^{-N+l-1}\lambda_k)\widehat{a}_{0}^*(2^{-N+l-2}\lambda_k)\cdots\widehat{a}_{0}^*(2^{-N}\lambda_k)\widehat{{f_G}}[k] & 2\le l\le L.
\end{cases}
\end{equation}
The index $j$ denotes the band of the transform with $j=0$ the low frequency component and $1\le j\le r$ the high frequency components. The index $l$ denotes the level of the transform.

Given a graph function $f_G$, let $\bm{\alpha}:=\bm{W}f_G:=\{\alpha_{j,l}:\ (j,l)\in\mathbb{B}\}$, with $\alpha_{j,l}:=W_{j,l}{f_G}$, be its tight wavelet frame coefficients. We denote the \textit{discrete tight wavelet frame reconstruction} as $\bm{W}^\top\bm{\alpha}$, which is defined by the following iterative procedure in frequency domain
\begin{equation}\label{D:TightFrame:Discrete:Reconstruction}
\begin{split}
\widehat\alpha_{0,l-1}[k]=\sum_{j=0}^r\widehat a_{j}(2^{-N+l-1}\lambda_k)\widehat{\alpha}_{j,l}[k]\qquad\mbox{for } l=L, L-1, \ldots, 1,
\end{split}
\end{equation}
where $\alpha_{0,0}:=\bW^\top\bm{\alpha}$ is the reconstructed graph data from $\bm{\alpha}$. Note that $\bW$ is obviously a linear transformation, and it is easy to verify that the linear transformation $\bW^\top$ defined by \eqref{D:TightFrame:Discrete:Reconstruction} is indeed the adjoint of $\bW$ satisfying $\langle \bW f_G, \bm{\alpha}\rangle=\langle f_G, \bW^\top\bm{\alpha}\rangle$ for all $f_G$ and $\bm{\alpha}$.

Perfect reconstruction from $\bm{\alpha}$ to $f_G$ through $\bm{W}^\top$ can be verified as long as the masks $\{a_j:\ 0\le j\le r\}$ satisfy the condition \eqref{UEP:1} in Theorem \ref{Theorem:TightFrame:Manifold}. We have the following theorem stating that $\{a_j: 0\le j\le r\}$ generates a discrete tight wavelet frame on graphs.
\begin{theorem}\label{Theorem:TightFrame:Graph}
Given a set of masks $\{a_j: 0\le j\le r\}\subset\ell_0(\Z)$, suppose condition \eqref{UEP:1} is satisfied. Then, the discrete tight wavelet frame transforms $\bm{W}$ and $\bm{W}^\top$ defined on $G=\{E, V, w\}$ by \eqref{D:TightFrame:Discrete} and \eqref{D:TightFrame:Discrete:Reconstruction} satisfy $$\bm{W}^\top \bm{W}f_G=f_G,\quad\mbox{for all } f_G: V\mapsto \R.$$
\end{theorem}
\begin{proof}
For simplicity, we prove the theorem for the case $L=2$. The proof for general $L$ is analogous. For $L=2$, we have $\bm{\alpha}:=\left\{\alpha_{j,l}: (j,l)\in\mathbb{B}\right\}$ with $\alpha_{j,l}:=W_{j,l}{f_G}$ and $\mathbb{B}=\{(j,l) : 1\le j\le r, l=1,2\}\cup\{(0,2)\}$.
By \eqref{D:TightFrame:Discrete}, we have $$\widehat\alpha_{j,2}[k]=\widehat{a}_{j}^*(2^{-N+1}\lambda_k)\widehat{a}_{0}^*(2^{-N}\lambda_k)\widehat{{f_G}}[k].$$ Then, letting $l=2$ in \eqref{D:TightFrame:Discrete:Reconstruction}, we have
\begin{equation*}
\begin{split}
\widehat\alpha_{0,1}[k]&=\sum_{j=0}^r\widehat a_{j}(2^{-N+1}\lambda_k)\widehat{\alpha}_{j,2}[k]\cr
 &=\sum_{j=0}^r\left|\widehat a_{j}(2^{-N+1}\lambda_k)\right|^2\widehat{a}_{0}^*(2^{-N}\lambda_k)\widehat{{f_G}}[k]\cr
 &=\widehat{a}_{0}^*(2^{-N}\lambda_k)\widehat{{f_G}}[k],
\end{split}
\end{equation*}
where the last identity follows from \eqref{UEP:1}. Letting $l=1$ in \eqref{D:TightFrame:Discrete:Reconstruction} and similarly, we have
\begin{equation*}
\begin{split}
\widehat\alpha_{0,0}[k]&=\sum_{j=0}^r\widehat a_{j}(2^{-N}\lambda_k)\widehat{\alpha}_{j,1}[k]\cr
 &=\sum_{j=0}^r\left|\widehat a_{j}(2^{-N}\lambda_k)\right|^2\widehat{{f_G}}[k]\cr
 &=\widehat{{f_G}}[k].
\end{split}
\end{equation*}
This shows that $\bW^\top\bW f_G=f_G$.
\end{proof}

We present some sets of masks satisfying \eqref{UEP:1}, and hence generate discrete tight wavelet frames on graphs according to Theorem \ref{Theorem:TightFrame:Graph}. Note that these masks are modified from those given by \eqref{Masks:BSplineTightFramelets} by removing the complex factors.

\begin{example}\label{example:masks}
We present three sets of masks that satisfy \eqref{UEP:1} and hence generate discrete tight wavelet frames on graphs. We shall refer to them as the ``Haar'', ``linear'' and ``quadratic'' tight wavelet frame systems, and the masks as the ``Haar'', ``linear'' and ``quadratic'' masks. These masks slice the spectrum in a different but similar way as shown in Figure \ref{Fig:Masks}.
\begin{enumerate}
\item \textbf{Haar.} $$\widehat a_0(\xi) = \cos(\xi/2)\quad\mbox{and}\quad \widehat a_1(\xi)=\sin(\xi/2).$$
\item \textbf{Linear.} $$\widehat a_0(\xi) = \cos^2(\xi/2),\quad \widehat a_1(\xi)=\frac{1}{\sqrt2}\sin(\xi)\quad\mbox{and}\quad \widehat a_2(\xi)=\sin^2(\xi/2).$$
\item \textbf{Quadratic.}
 \begin{equation*}
 \begin{split}
  &\widehat a_0(\xi) = \cos^3(\xi/2),\quad \widehat a_1(\xi)=\sqrt{3}\sin(\xi/2)\cos^2(\xi/2)\\
  &\widehat a_2(\xi)=\sqrt{3}\sin^2(\xi/2)\cos(\xi/2)\quad\mbox{and}\quad \widehat a_3(\xi)=\sin^3(\xi/2).
 \end{split}
 \end{equation*}
\end{enumerate}
\end{example}

\begin{figure}[htp]
\centering
    \includegraphics[width=6.0in]{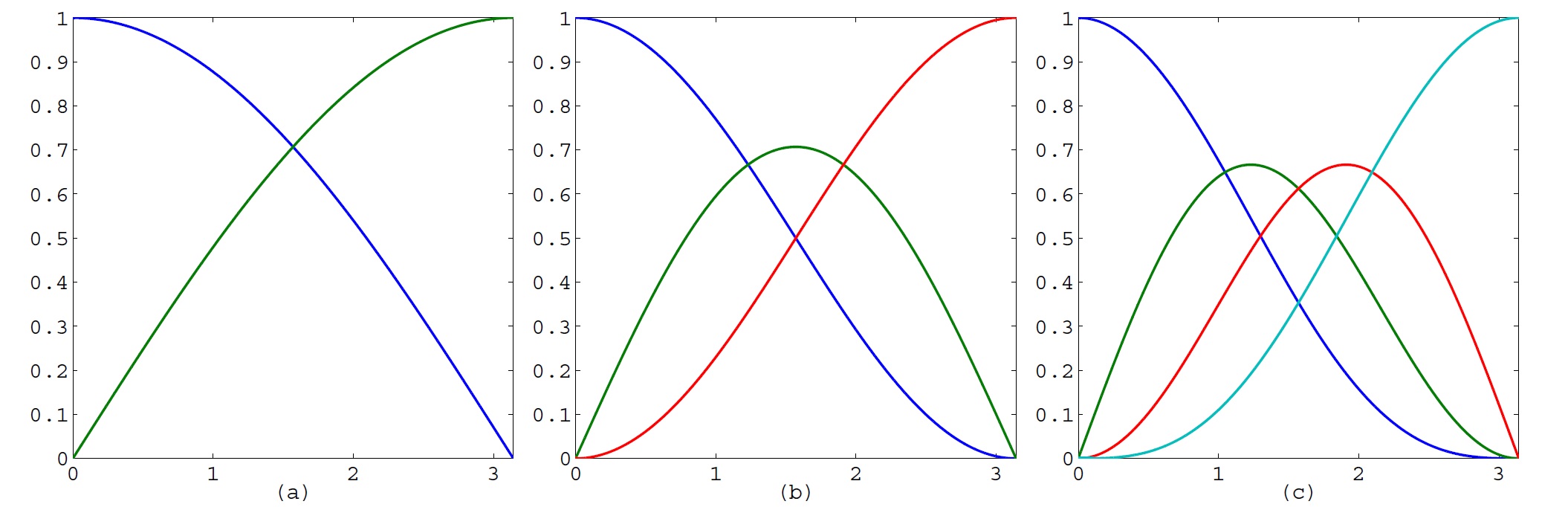}
\caption{Plots of $\widehat a_j(\xi)$ for $\xi\in[0, \pi]$. Plot (a) presents the ``Haar'' masks with $\widehat a_0$ in blue and $\widehat a_1$ in green. Plot (b) presents the ``linear'' masks with $\widehat a_0$ in blue, $\widehat a_1$ in green and $\widehat a_2$ in red. Plot (c) presents the ``quadratic'' masks with $\widehat a_0$ in blue, $\widehat a_1$ in green, $\widehat a_2$ in red and $\widehat a_3$ in light blue.}\label{Fig:Masks}
\end{figure}

\subsection{Fast Tight Wavelet Frame Transform on Graphs (WFTG)}

The discrete tight wavelet frame transforms given by \eqref{D:TightFrame:Discrete} and \eqref{D:TightFrame:Discrete:Reconstruction} require the full set of eigenvectors and eigenvalues of the graph Laplacian, which is computationally expensive to obtain for large graphs. A solution to such computation challenge is to use polynomial approximation of the masks, such as the Chebyshev polynomials, so that eigenvalue decomposition of the graph Laplacian is not needed. This approach was first proposed by \cite{hammond2011wavelets}. However, the masks $\widehat{a}_j(\xi)$ that we use, as well as those constructed in the literature from the UEP, are trigonometric polynomials (such as \eqref{Masks:BSplineTightFramelets}). Therefore, $\widehat{a}_j(\xi)$ can be accurately approximated by \textit{low-degree} Chebyshev polynomials (see e.g. \cite{mason2010chebyshev}) which significantly reduces the computation cost of the decomposition and reconstruction algorithms. In this section, we describe the details of the fast tight wavelet frame transform on graphs (WFTG) based on polynomial approximation. Comparisons of our approach with that of \cite{hammond2011wavelets} will be presented later in Section \ref{Sec:WFTG} and Section \ref{Sec:Applications}.

We start with approximating a given function $g(\xi)$ with $\xi\in[0, \pi]$ using Chebyshev polynomials $\{T_k(\xi): k=0,1,\ldots\}$. Recall that the Chebyshev polynomials on $[0, \pi]$ are defined iteratively by
\begin{equation}\label{D:Chebyshev}
T_0=1,\quad T_1(\xi)=\frac{\xi-\pi/2}{\pi/2},\quad T_k(\xi)=\frac{4}{\pi}(\xi-\pi/2)T_{k-1}(\xi)-T_{k-2}(\xi),\ \mbox{for } k=2,3,\ldots.
\end{equation}
Then, given $g(\xi)$, we have the following expansion
\begin{equation*}
g=\frac12 c_{0}+\sum_{k=1}^\infty c_{k}T_k,\quad\mbox{where}\quad c_{k}=\frac2{\pi}\int_0^\pi\cos(k\theta)g\left(\frac{\pi}2(\cos(\theta)+1)\right)d\theta.
\end{equation*}
When $g(\xi)$ is smooth, we can use partial sums to accurately approximate $g(\xi)$:
$$g(\xi)\approx\cT^n(\xi)=\frac12 c_{0}+\sum_{k=1}^{n-1} c_{k}T_k.$$ We denote the Chebyshev polynomial approximation of $\widehat a_j(\xi)$ as $$\widehat a_j(\xi)\approx\cT_j^n(\xi)=\frac12 c_{j,0}+\sum_{k=1}^{n-1} c_{j,k}T_k,\qquad c_{j,k}=\frac2{\pi}\int_0^\pi\cos(k\theta)\widehat a_j\left(\frac{\pi}2(\cos(\theta)+1)\right)d\theta.$$ We denote the Chebyshev approximation of $\widehat a_j^*$ as $\cT_j^{n*}$. Since the masks $\{\widehat a_j\}$ we will be using are real-valued (those in Example \ref{example:masks}), we have $\cT_j^{n}=\cT_j^{n*}$.

Let us see how to speed up the transformation \eqref{D:TightFrame:Discrete} using the above approximation. The Laplacian $\cL$ admits the eigenvalue decomposition $\cL=U\Lambda U^\top$ where $\Lambda=\mbox{diag}\{\lambda_0,\lambda_1,\ldots,\lambda_{K-1}\}$ and columns of $U$ are the eigenvectors. Then we can rewrite \eqref{D:TightFrame:Discrete} in the following matrix form in physical domain:
\begin{equation}\label{D:TightFrame:Discrete:Physical}
{W}_{j,l}{f_G}=
\begin{cases}
U\widehat{a}_{j}^*(2^{-N}\Lambda)U^\top {f_G} & l=1,\\
U\widehat{a}_{j}^*(2^{-N+l-1}\Lambda)\widehat{a}_{0}^*(2^{-N+l-2}\Lambda)\cdots\widehat{a}_{0}^*(2^{-N}\Lambda)U^\top {f_G} & l\ge2.
\end{cases}
\end{equation}
Here,
\begin{equation*}
\widehat{a}_{j}^*(\gamma\Lambda):=\mbox{diag}\{\widehat{a}_{j}^*(\gamma\lambda_0),\widehat{a}_{j}^*(\gamma\lambda_1),\ldots,\widehat{a}_{j}^*(\gamma\lambda_{K-1})\}.
\end{equation*}
If we substitute the approximation $\widehat a_j(\xi)\approx\cT_j^n(\xi)$ in \eqref{D:TightFrame:Discrete:Physical} and use the fact that $\cT_j^n$ are polynomials, we obtain the WFTG:
\subsection*{Fast Tight Wavelet Frame Transform on Graphs (WFTG)} $\bm{W}f_G=\{W_{j,l}{f_G}: (j,l)\in\mathbb{B}\}$ where
\begin{equation}\label{D:TightFrame:Fast}
{W}_{j,l}{f_G}:=
\begin{cases}
\cT_j^{n*}(2^{-N}\cL) {f_G} & l=1,\\
\cT_j^{n*}(2^{-N+l-1}\cL)\cT_0^{n*}(2^{-N+l-2}\cL)\cdots \cT_0^{n*}(2^{-N}\cL) {f_G} & l\ge2.
\end{cases}
\end{equation}
Reconstruction transform $\bm{W}^\top$ can be defined similarly and we have $\bm{W}^\top\bm{W}\approx\bm{I}$. Note that, for computational efficiency, the operations $\cT_j^*(s\cL){f_G}$ for decomposition and $\cT_j(s\cL){f_G}$ for reconstruction are computed via the iterative definition of the Chebyshev polynomial \eqref{D:Chebyshev}. Therefore, in WFTG, only matrix-vector multiplications are involved.

\begin{remark}
As we mentioned earlier, since the masks $\widehat a_j(\xi)$ are trigonometric polynomials, they can be accurately approximated by low-degree Chebyshev polynomials (see Table \ref{Table:Chebyshev} for example). In our simulations, we choose $n=8$ which is sufficient for the applications we considered. Note that, if a higher order B-spline tight wavelet frame system is used, Chebyshev polynomials with a higher degree may be needed to achieve a given approximation accuracy. However, just as in image processing, only the lower order systems are mostly used because they provide a better balance between quality and computation efficiency.
\end{remark}

\subsection{Numerical Simulations of WFTG}\label{Sec:WFTG}

Given a graph $G=\{E, V, w\}$ with $|V|=K$, we shall use the unnormalized graph Laplacian $\cL=D-A$ as an example. We compute only the largest eigenvalue $\lambda_{K-1}$, and choose the initial dilation scale as $N=\log_2(\frac{\lambda_{K-1}}{\pi})$.

In our simulations, we choose the ``linear'' tight wavelet frame system given in Example \ref{example:masks}, i.e.
$$\widehat a_0(\xi)=\cos^2(\xi/2),\quad \widehat a_1(\xi)=\frac{1}{\sqrt2}\sin(\xi)\quad\mbox{and}\quad \widehat a_2(\xi)=\sin^2(\xi/2).$$ We present approximation errors of $\cT_j^n$ to $\widehat a_j$ for several values of $n$ in the following Table \ref{Table:Chebyshev}, where one can easily see that the $\widehat a_j$ can be accurately approximated by low-degree Chebyshev polynomials. In our simulations, we fix $n=8$, i.e. we use Chebyshev polynomials of degree 7 to approximate the masks given in Example \ref{example:masks}.

\begin{table}[htp]
\caption{{Approximation error $\|\cT_j^n-\widehat a_j\|_\infty$ with $\widehat a_j$ the `linear'' masks given in Example \ref{example:masks}.}}\label{Table:Chebyshev}
\centering
\begin{tabular}{c|c|c|c}
\hline
\multicolumn{1}{c|}{Errors} & \multicolumn{1}{c|}{$\widehat a_0$} & \multicolumn{1}{c|}{$\widehat a_1$} &
\multicolumn{1}{c}{$\widehat a_2$}\\
\hline
 \multicolumn{1}{c|}{$n=4$}         & 2.273$\times 10^{-3}$ & 2.022$\times 10^{-2}$ & 2.273$\times 10^{-3}$ \\
 \multicolumn{1}{c|}{$n=5$}         & 2.273$\times 10^{-3}$ & 4.267$\times 10^{-4}$ & 2.273$\times 10^{-3}$ \\
 \multicolumn{1}{c|}{$n=6$}         & 3.417$\times 10^{-5}$ & 4.267$\times 10^{-4}$ & 3.417$\times 10^{-5}$ \\
 \multicolumn{1}{c|}{$n=7$}         & 3.417$\times 10^{-5}$ & 4.775$\times 10^{-6}$ & 3.417$\times 10^{-5}$ \\
 \multicolumn{1}{c|}{$n=8$}         & 3.762$\times 10^{-7}$ & 4.775$\times 10^{-6}$ & 3.762$\times 10^{-7}$ \\
 \hline
\end{tabular}
\end{table}

To illustrate the effects of the WFTG using the ``linear'' masks in Example \ref{example:masks}, we consider a graph $G=\{E, V, w\}$, with $V$ being sampled from a unit sphere. The number of vertices in $V$ is $16,728$. The adjacency matrix $A$ is generated using the weight function $w(v_i,v_j)=e^{-\|v_i-v_j\|_2^2/\sigma}$ for $v_i,v_j\in V$. In our experiments, we fix $\sigma=10$. Then we threshold $A$ to limit the number of nearest neighbors of each vertex to 10. The functions ${f_G}: V\mapsto \R$ are generated by mapping two images, ``Slope'' and ``Eric'', onto the graph $G$ (see Figure \ref{Fig:WFTG:GraphData}). We perform 4 levels of WFTG and use $n=8$ for the Chebyshev polynomial approximation of the masks. The tight wavelet frame coefficients ${W}_{j,l}f_G$ for $0\le j\le 2$ and $1\le l\le 4$ are shown in Figure \ref{Fig:WFTG:Decomp:Slope} and \ref{Fig:WFTG:Decomp:Eric}. It is worth noticing that the high frequency coefficients ${W}_{j,l}f_G$ for $j=1,2$ and $1\le l\le 4$ are very sparse, which is consistent with what we normally observe from wavelet frame coefficients of images. In other words, the proposed tight wavelet frames provide a \textit{sparse} representation for graph functions. Sparsity is crucial to many applications. In the next section, we shall consider denoising and semi-supervised clustering utilizing the sparse representation provided by the proposed WFTG.

\begin{figure}[htp]
\centering
    \includegraphics[width=4.0in]{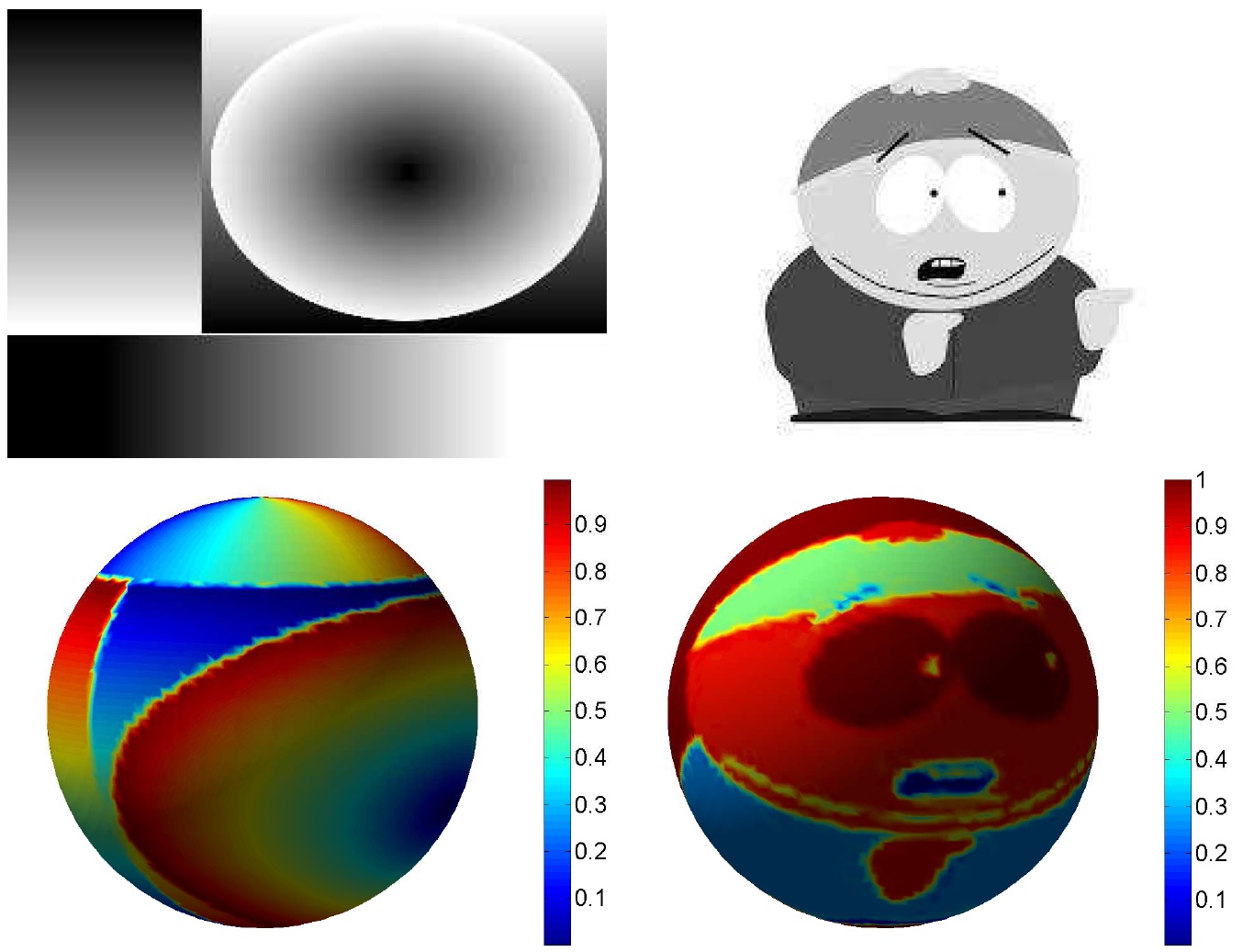}
\caption{This figure shows two images (first row), ``Slope'' and ``Eric'', that are mapped to the graph of the unit sphere to form the graph data ${f_G}$ (second row).}\label{Fig:WFTG:GraphData}
\end{figure}

\begin{figure}[htp]
\centering
    \includegraphics[width=6.0in]{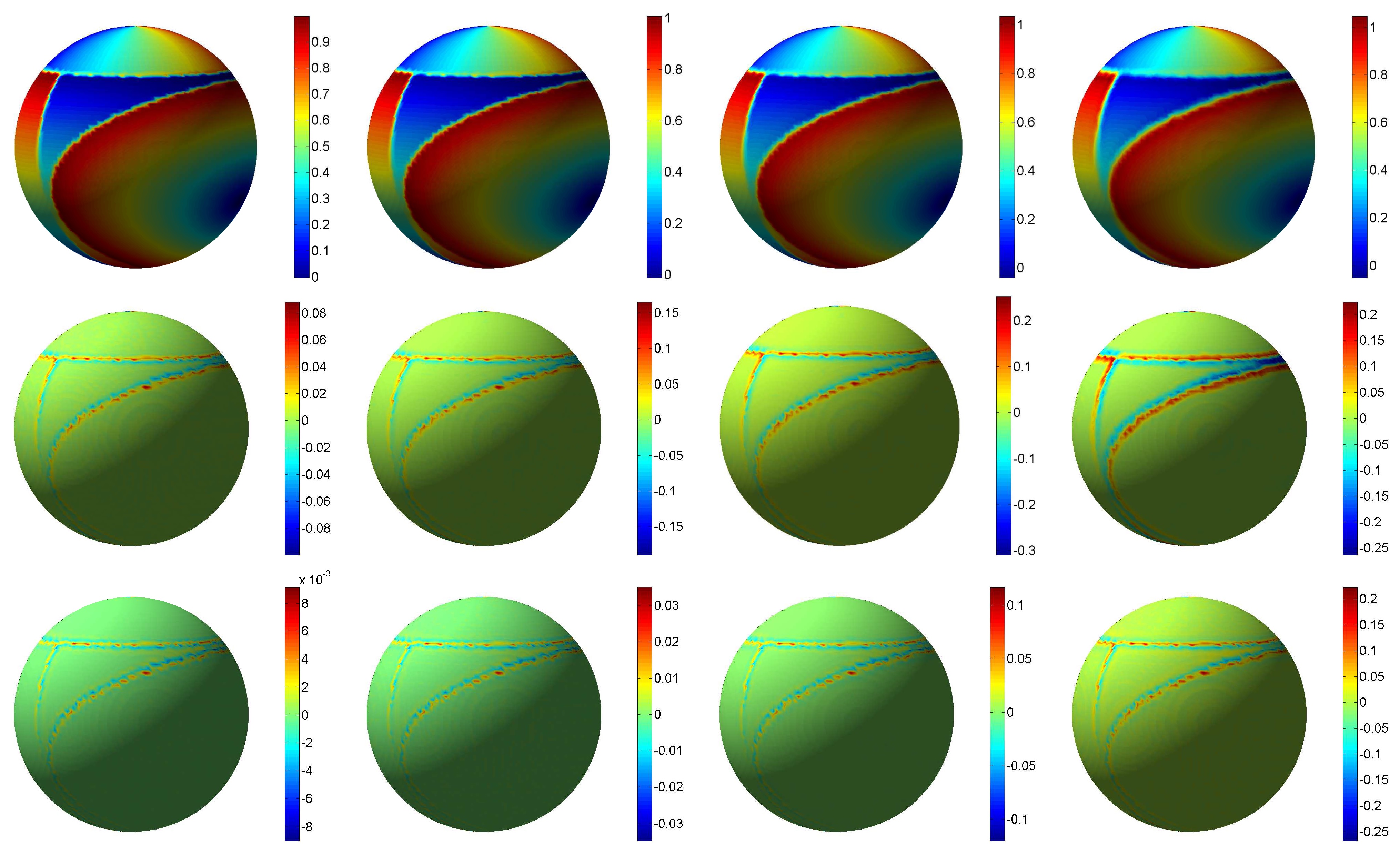}
\caption{This figure shows the images of the tight wavelet frame coefficients ${W}_{j,l}f_G$ for $0\le j\le 2$ (row 1-3) and $1\le l\le 4$ (column 1-4).}\label{Fig:WFTG:Decomp:Slope}
\end{figure}

\begin{figure}[htp]
\centering
    \includegraphics[width=6.0in]{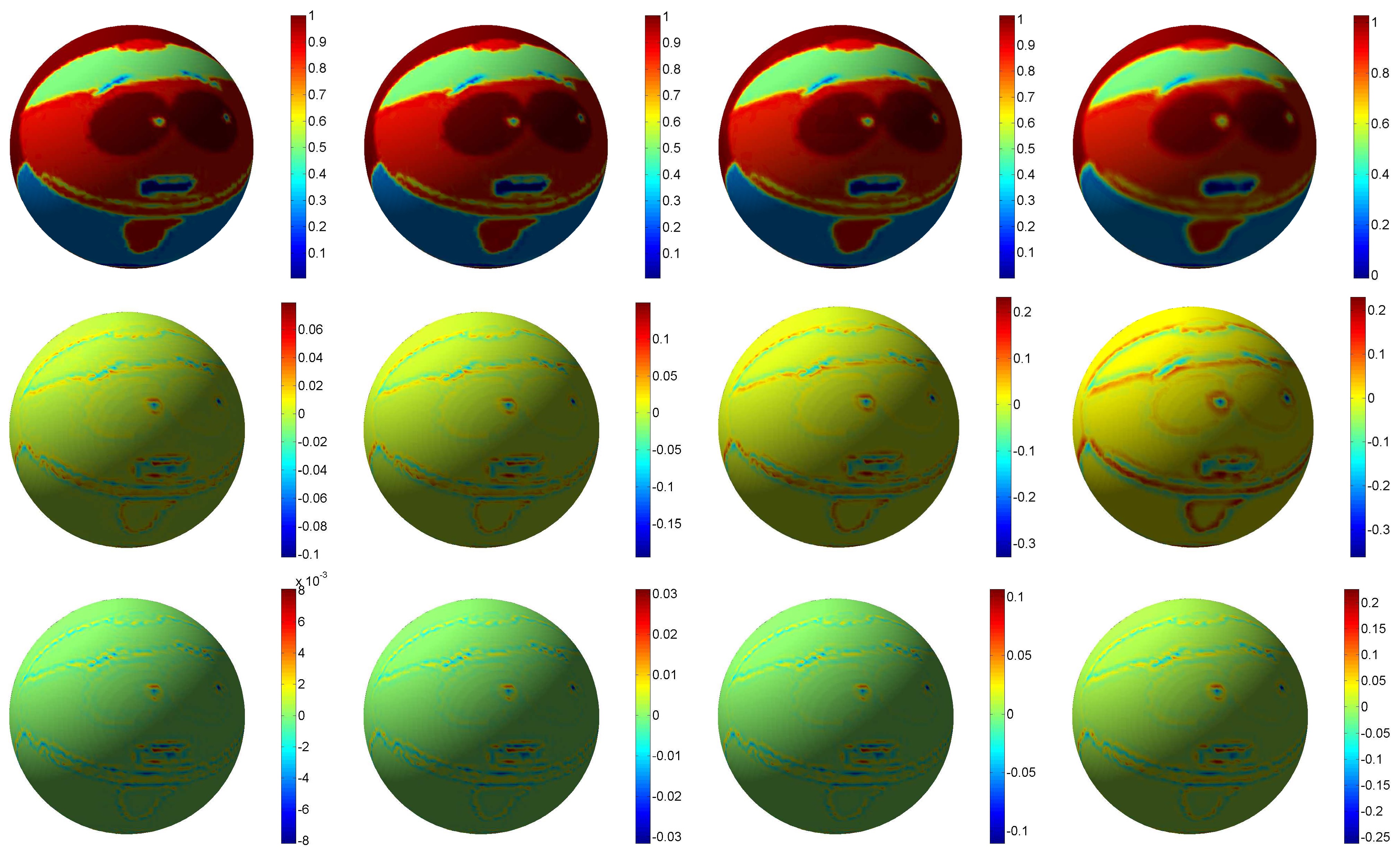}
\caption{This figure shows the images of the tight wavelet frame coefficients ${W}_{j,l}f_G$ for $0\le j\le 2$ (row 1-3) and $1\le l\le 4$ (column 1-4).}\label{Fig:WFTG:Decomp:Eric}
\end{figure}

Finally, we compare the computation efficiency of the WFTG with that of the spectral graph wavelet transform (SGWT) of \cite{hammond2011wavelets}, since SGWT is most related to the proposed WFTG. All simulations are conducted on a laptop with Intel(R) Core(TM) i7-4600U CPU (2.1GHz \& 2.7GHz) and 8GB RAM running Windows 7 (64-bit). The masks used for WFTG are those given in Example \ref{example:masks}. The codes for SGWT is downloaded from ``http://wiki.epfl.ch/sgwt". The number of scales of SGWT is chosen to be 4 and order of polynomial approximation is chosen to be 25 (by default). The specific graph wavelet used for SGWT is generated by a cubic spline (see \cite{hammond2011wavelets} for detail).

The computation time and reconstruction errors of a 4-level decomposition and reconstruction of the WFTG and SGWT are summarized in Table \ref{Table:Comparison:Efficiency}. For a similar reconstruction error, the WFTG using ``Haar" masks is faster than the SGWT; the WFTG using ``linear" masks is comparable in speed to SGWT; and the WFTG with ``quadratic" masks is slower. However, the WFTG is a more redundant transformation than the SGWT in general. For example, after 4 levels of decomposition, the size of the data in transform domain is 5 times of the size of the original data for ``Haar", 9 times for ``linear" and 13 times for ``quadratic", whereas the size of the data in the transform domain of SGWT is 5 times the size of the original data. Therefore, taking the amount of redundancy into account, WFTG is in fact faster than the SGWT in general. In Section \ref{Sec:Applications}, we will further show that the WFTG is more effective than the SGWT in graph data denoising and semi-supervised clustering, which indicates that the additional redundancy by the WFTG is beneficial in applications.

In practice, we often need to properly balance between computation efficiency and redundancy when choosing a system to use for a specific application. Higher redundancy may lead to better results while sacrificing efficiency. Taking the denoising problem in Section \ref{SubS:Denoising} as an example, we can see from Table \ref{Table:Denoising} that using ``quadratic" does not significantly improve denoising quality while we generally face a 50\% increase in total computation time. Therefore, taking both efficiency and redundancy into account, one may stick to ``Haar" and ``Linear" or other set of masks with comparable redundancy in applications.

The computation efficiency of the WFTG is mainly due to the efficient approximation of the masks by Chebyshev polynomial, because the masks constructed by our method are analytic which leads to geometric decay of the coefficients of the Chebyshev polynomial approximation (see e.g. \cite{mason2010chebyshev}). In other words, the masks $\widehat{a}_j(\xi)$ can be accurately approximated by \textit{low-degree} Chebyshev polynomials. In addition, since tight frames are used by WFTG, the reconstruction step is straightforward, and we do not need to solve a least-squares problem to find the inverse as what is used by the SGWT.

\begin{table}[htp]
\tiny\caption{{Comparison of computation efficiency of WFTG and SGWT \cite{hammond2011wavelets}. The reconstruction time is in seconds and errors are measured by the $\ell_\infty$-norm.}}\label{Table:Comparison:Efficiency}
\centering
\begin{tabular}{|l||c c|c c|c c|c c|}
\hline
\multicolumn{1}{|l||}{Image} &
\multicolumn{2}{c|}{WFTG (Haar)} & \multicolumn{2}{c|}{WFTG (Linear)} & \multicolumn{2}{c|}{WFTG (Quadratic)} & \multicolumn{2}{c|}{SGWT}
\\
\multicolumn{1}{|l||}{Name} &
\multicolumn{1}{c}{Time (sec.)} & \multicolumn{1}{c|}{Error} &
\multicolumn{1}{c}{Time (sec.)} & \multicolumn{1}{c|}{Error} &
\multicolumn{1}{c}{Time (sec.)} & \multicolumn{1}{c|}{Error} &
\multicolumn{1}{c}{Time (sec.)} & \multicolumn{1}{c|}{Error}
\\
\hline
Slope    & 0.246  & $3.80\times 10^{-6}$ & 0.450  & $7.16\times 10^{-6}$ & 0.773  & $3.80\times 10^{-6}$ & 0.519  & $5.02\times 10^{-6}$ \\
Eric     & 0.262  & $3.81\times 10^{-6}$ & 0.462  & $7.63\times 10^{-6}$ & 0.782  & $3.81\times 10^{-6}$ & 0.509  & $1.21\times 10^{-6}$ \\
\hline
\end{tabular}
\end{table}

\section{Applications: Denoising and Semi-Supervised Clustering}\label{Sec:Applications}

In the previous section, we discussed how we can define sparse representations on graphs by tight wavelet frames, and how to efficiently compute the decomposition and reconstruction transforms using polynomial approximation of the masks. There are many potential applications of the proposed representation. For example, one may consider a general linear inverse problem on graphs whenever the solution we seek is sparse under the proposed tight wavelet frame transform. In this section, we shall focus on the applications of the WFTG to graph data denoising and semi-supervised clustering.

Given a graph data $f$, its Euclidean $\ell_p$-norm ($p=1$ or 2) is defined by $$\|f\|_p:=\left(\sum_{k=1}^{K}|f[k]|^p\right)^{\frac1p},$$ and its graph-$\ell_p$-norm is defined by $$\|f\|_{p,G}:=\left(\sum_{k=1}^{K}|f[k]|^pd[k]\right)^{\frac1p},$$ were $d[k]$ is the degree of the node $v_k\in V$. Let $D:=\mbox{diag}\{d[1], d[2], \ldots, d[K]\}$. Obviously, we have $\|D^{1/p}f\|_p=\|f\|_{p,G}$ with $D^s:=\mbox{diag}\{d^s[1], d^s[2], \ldots, d^s[K]\}$ for $s\in\R$. Note that we have dropped the subscript ``$G$" of $f_G$ that we used earlier to differentiate functions on graphs from functions on manifolds, since we shall stay in the discrete setting throughout this section.

What is crucial to the success of the proposed WFTG in solving denoising and clustering problems is the sparse approximation provided by the underlying tight frame system. Because of this, the models we propose for both of the problems take the following generic form
\begin{equation}\label{Model:Generic}
\min_{u}\ \nu\|\bW u\|_{1,G}+F(u,f),
\end{equation}
where $F(\cdot,\cdot)$ is some fidelity function that takes different forms for different applications, and $f$ is the observed graph data. The use of the $\ell_1$-norm of $\bW u$ is because the desirable solution we seek can be sparsely approximated by the underlying wavelet frame system.

Note that model \eqref{Model:Generic} resembles some variational models used in the literature for image and graph data processing and analysis, such as the total variation (TV) based models \cite{ROF,Gilboa2008,GO,bertozzi2012diffuse,merkurjev2013mbo,merkurjev2014global}. However, the modeling philosophy of \eqref{Model:Generic}, as well as most wavelet frame based models for image processing, is based on sparsity of the coefficients in transform domain, whereas for variational models such as TV models, data to be recovered are characterized by carefully chosen function spaces. Nonetheless, comparing most wavelet (frame) based models and variational models, it seems that the major difference between them is the choice of the underlying transformation that maps the data to a certain domain where it can be sparsely approximated.

On flat domains, fundamental connections between wavelet frame based approach (taking the form of \eqref{Model:Generic} with a specific fidelity function $F(\cdot,\cdot)$) and variational methods were established in \cite{CDOS2011,CDS2014}. In particular, connections to the total variation model \cite{ROF} was established in \cite{CDOS2011}, and to the Mumford-Shah model \cite{mumford1989optimal} was established in \cite{CDS2014}. Furthermore, in \cite{DJS2013}, the authors established a generic connection between iterative wavelet frame shrinkage and general nonlinear evolution PDEs. The series of three papers \cite{CDOS2011,DJS2013,CDS2014} showed that wavelet frame transforms are discretizations of differential operators in both variational and PDE frameworks, and such discretization is superior to some of the traditional finite difference schemes for image restoration.

One interesting question to ask is whether \eqref{Model:Generic} is a discrete form of a certain variational model, or in other words, whether the WFTG discretizes differential operators on manifolds? If so, what type of convergence can we establish between the discrete model \eqref{Model:Generic} and its corresponding variational model? We do not attempt to answer these questions in this paper since a rigorous justification may be rather technical. But we observe from Figure \ref{Fig:WFTG:Decomp:Slope} and \ref{Fig:WFTG:Decomp:Eric} that the WFTG does behave like differential operators.

Finally, we note that the quality of model \eqref{Model:Generic} for a given application depends on the quality of the transformation $\bW$. Taking denoising as an example, redundant systems are often better than non-redundancy systems. Furthermore, two different redundant systems may produce very different results even when the same model as \eqref{Model:Generic} is used. In this section, we will compare the WFTG and the SGWT for graph denoising and graph clustering problems. Our results will show that WFTG is significantly more effective than the SGWT at least for denoising and clustering problems. Furthermore, we will compare our proposed graph clustering model based on the WFTG with some state-of-the-art clustering models using two real data sets.

\subsection{Denoising}\label{SubS:Denoising}

Denoising on graphs may not be a problem as important as image denoising. However, a good representation together with the modeling based on such representation should be robust to noise in order for them to be useful in practice. Therefore, denoising is a necessary test problem. On the other hand, if the graph data collected in practice is corrupted by a certain type of noise that needs to be removed, denoising is a helpful step to take before any further data analysis. Given a graph $G=\{E, V, w\}$ with $|V|=K$, and a graph function $\bar u: V\mapsto \R$, the observed data or the data we collect is $f=\bar u+\eta$, where $\eta$ is some noise, which shall be assumed to be Gaussian white noise.

Now, we propose to use the following analysis based model for denoising, which was originally proposed in image processing \cite{elad2005simultaneous,starck2005image,cai2009split}:
\begin{equation}\label{Model:Denoising}
\min_{u}\ \|\bnu\cdot\bm{W}u\|_{1,G}+\frac12\|u-f\|_{2,G}^2,
\end{equation}
where $\bm{W}$ is the tight wavelet frame transform \eqref{D:TightFrame:Discrete} and $$\|\bnu\cdot\bm{W}u\|_{1,G}:=\sum_{(j,l)\in\mathbb{B}}\nu_{j,l}\|W_{j,l}u\|_{1,G}$$ with $\nu_{j,l}\ge0$ being the tuning parameters and $\mathbb{B}$ defined in \eqref{D:WT:IndexSet}. In our simulations, we choose
\begin{equation}\label{D:Lambda}
\nu_{j,l}=
\begin{cases}
4^{-l+1}\nu & \mbox{for } j\ne0, 1\le l\le L,\\
0 & \mbox{for } j=0, l=L,
\end{cases}
\end{equation}
where $\lambda>0$ is a scalar tuning parameter. Note that model \eqref{Model:Denoising} can be solved efficiently using the split Bregman algorithm \cite{GoldO,cai2009split}, which is also equivalent to the alternating direction method of multipliers (ADMM)
\cite{gabay1976dual,bertsekas1989parallel,eckstein1992douglas}. Applying the derivation of split Bregman algorithm to our model \eqref{Model:Denoising}, we have
\begin{equation}\label{Algorithm:ADMM:Denoising}
\begin{cases}
u^{k+1}=\arg\min_{u}\ \frac12\|u-f\|_{2,G}^2+\frac{\mu}{2}\|\bm{W}u-d^k+b^k\|_2^2,\\
d^{k+1}=\arg\min_{d}\ \|\bnu\cdot d\|_{1,G}+\frac{\mu}{2}\|d-(\bm{W}u^{k+1}+b^k)\|_2^2,\\
b^{k+1}=b^{k}+\bm{W}u^{k+1}-d^{k+1}.
\end{cases}
\end{equation}
Both of the sub-optimization problems in \eqref{Algorithm:ADMM:Denoising} have closed-form solutions. Replacing them by their closed-form solutions, we obtain the following denoising algorithm solving \eqref{Model:Denoising}:
\begin{equation}\label{Algorithm:Denoising}
\begin{cases}
u^{k+1}=(D+\mu I)^{-1}\left(Df+\mu\bm{W}^\top(d^k-b^k)\right),\\
d^{k+1}=\cS_{\bnu_G/\mu}(\bm{W}u^{k+1}+b^k),\\
b^{k+1}=b^{k}+\bm{W}u^{k+1}-d^{k+1},
\end{cases}
\end{equation}
where $(D+\mu I)^{-1}=\mbox{diag}\{\frac{1}{d[1]+\mu},\ldots,\frac{1}{d[K]+\mu}\}$, $\bnu_G:=\{\nu_{j,l}D:\ (j,l)\in\mathbb{B}\}$ and the soft-thresholding operator $\cS$ is defined pointwise as $\cS_x(y):=\frac{y}{|y|}\max\{|y|-x,0\}$. In our numerical implementation, we choose zero initialization of the algorithm, i.e. $u^0=0$ and $b^0=d^0=0$.

\begin{remark}
We specifically choose the Euclidean $\ell_2$-norm in the second term of the first two equations of \eqref{Algorithm:ADMM:Denoising}, which corresponds to the choice of the Euclidean linear and augmented terms in the derivation of the split Bregman algorithm using the augmented Lagrangian method (see \cite{gabay1976dual,bertsekas1989parallel,eckstein1992douglas} or \cite[Chapter 4]{Dong2010IASNotes} for details). Another choice, perhaps a more straightforward choice, is to use $\|\cdot\|_{2,G}^2$. However, this will lead to an inversion of the non-diagonal matrix $D+\bm{W}^\top \bm{D}\bm{W}$ in solving for $u^{k+1}$, where $\bm{D}$ is a $(rL+1)K\times(rL+1)K$ diagonal matrix given by $\bm{D}=\mbox{diag}\{D, D, \ldots, D\}$. Therefore, for computation efficiency, we shall use the algorithm \eqref{Algorithm:Denoising} to solve \eqref{Model:Denoising}.
\end{remark}

The graph we consider is the one formed by sampling a unit sphere as described in Section \ref{Sec:WFTG}. We select four examples of graph data as the noise-free data $\bar u$, where ``Slope'' and ``Eric'' were given in Figure \ref{Fig:WFTG:GraphData}, and the other two graph data, ``Barbara'' and ``Lena'', are generated by mapping these two widely used images on the discrete sphere. Gaussian white noise is added to $\bar u$ with a standard deviation $0.05$. The noise-free and noisy data are visualized in Figure \ref{Fig:Denoising:Linear}, first and second row respectively. The only reason that spherical data is used instead of general graph data is for visualization purpose, so that one can visually judge the denoising quality by examining the images.

In our simulations, when the WFTG is chosen for $\bW$, we fix the level of decomposition to 1, since using higher decomposition levels only slightly improves the denoising quality while the computation cost is noticeably increased. We choose ``Haar", ``linear" and ``quadratic" masks given by Example \ref{example:masks} for the WFTG and compare the denoising results. For each choice of masks and each graph data, the parameter $\nu$ of \eqref{D:Lambda} is adjusted manually for optimal denoising quality.

For comparison purpose, we also choose the SGWT as the transform $\bW$ in model \eqref{Model:Denoising}. The implementation of the algorithm \eqref{Algorithm:ADMM:Denoising} is exactly the same as using WFTG for $\bW$. For SGWT, we choose the number of scales to be 5 and order of the Chebyshev polynomial approximation as 25. Note that one may choose a different number of scales. However, the denoising quality varies only slightly when larger number of scales are chosen. The specific graph wavelet systems of SGWT used for the comparisons are generated by a cubic spline and the Mexican hat wavelet (see \cite{hammond2011wavelets} for further details). For optimal denoising quality, we select the same parameter $\nu$ across all different scales (instead of scaling it like in \eqref{D:Lambda}), where the specific value of $\nu$ is manually chosen for different graph data.

To see the denoising effect using the WFTG and SGWT with different frame systems, we summarize the denoising errors in Table \ref{Table:Denoising}. The desnoising error is defined by $$\mbox{Denoising Error }=\frac{\|u-\bar u\|_2}{\|\bar u\|_2},$$ where $u$ is the denoised data and $\bar u$ is the original data. From Table \ref{Table:Denoising}, it is obvious that the denoising results using the WFTG are significantly better than those of the SGWT. For WFTG, we can see that ``linear'' produces better results than ``Haar'', while ``quadratic'' is comparable with ``linear''. However, WFTG using ``Quadratic'' masks is slower than ``Haar'' and ``Linear''. Therefore, the ``linear'' tight wavelet frame system seems to be a good balance between computational cost and reconstruction quality for denoising problems. However, when the graph data can be well approximated by a piecewise constant function, such as ``Slope'', the ``Haar'' system seems to be the best choice. For visual examination, denoising results using the ``linear'' tight wavelet frame system for the WFTG, cubic spline and Mexican hat wavelet for SGWT are shown in the 3rd-5th row of Figure \ref{Fig:Denoising:Linear}. It is obvious from Figure \ref{Fig:Denoising:Linear} that the visual quality of the denoising results of the WFTG is significantly better than that of the SGWT.

Finally, we note that if one has a linear inverse problem to solve on graph, i.e. $f=A\bar u+\eta$, then we can change the second term of \eqref{Model:Denoising} to $\frac12\|Au-f\|_{2,G}^2$ (see e.g. \cite{cai2009split,CDOS2011}). Also, if the additive noise $\eta$ is non-Gaussian, such as Poisson, one may modify the norm of the second term of \eqref{Model:Denoising} to properly incorporate the correct noise statistics (see e.g. \cite{le2007variational,chan2007multilevel,zhang2008wavelets}).

\begin{table}[htp]
\caption{{Comparisons: denoising errors of WFTG and SGWT.}}\label{Table:Denoising}
\centering
\begin{tabular}{c|c|c|c|c|c}
\hline
\multicolumn{1}{c|}{Errors} & \multicolumn{1}{c|}{WFTG} & \multicolumn{1}{c|}{WFTG} &
\multicolumn{1}{c|}{WFTG} & \multicolumn{1}{c|}{SGWT} & \multicolumn{1}{c}{SGWT}\\
\multicolumn{1}{c|}{} & \multicolumn{1}{c|}{``Haar''} & \multicolumn{1}{c|}{``Linear''} &
\multicolumn{1}{c|}{``Quadratic''} & \multicolumn{1}{c|}{``Cubic"} & \multicolumn{1}{c}{``Mexican Hat"}\\
\hline
 \multicolumn{1}{c|}{``Slope''}        & 0.04279 & 0.04224 &  0.04307 & 0.06278 & 0.05989 \\
 \multicolumn{1}{c|}{``Eric''}         & 0.02995 & 0.02872 &  0.02886 & 0.03926 & 0.03787 \\
 \multicolumn{1}{c|}{``Barbara''}      & 0.07587 & 0.07419 &  0.07376 & 0.08857 & 0.08743 \\
 \multicolumn{1}{c|}{``Lena''}         & 0.07972 & 0.07864 &  0.07829 & 0.08692 & 0.08830 \\
 \hline
\end{tabular}
\end{table}

\begin{figure}[htp]
\centering
    \includegraphics[width=5.6in]{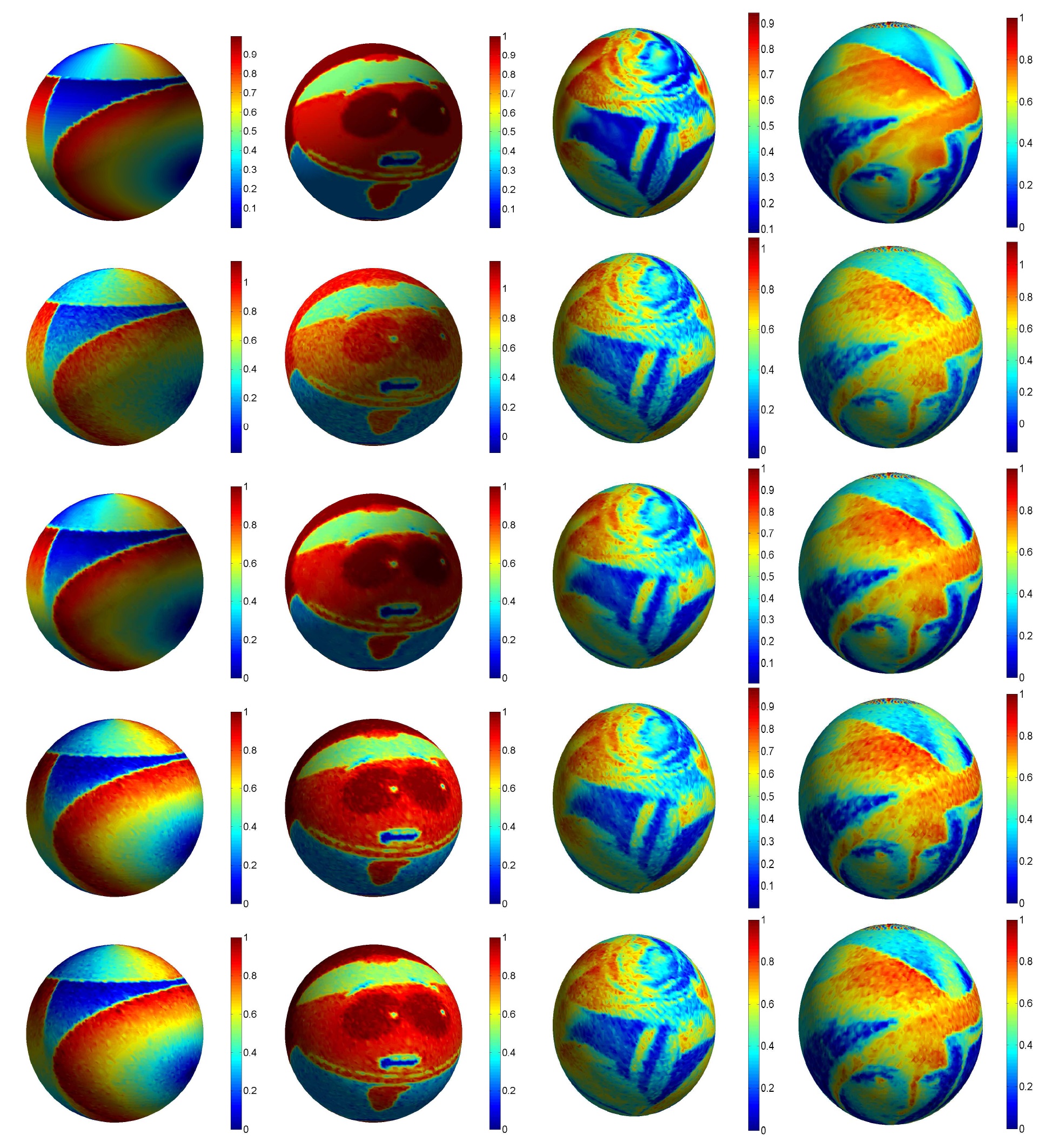}
\caption{This figure shows the four noise-free graph data (first row), the corresponding noisy data (second row) and the denoising results (3rd-5th row) of  WFTG (``linear"), SGWT (cubic spline) and SGWT (Mexican hat).}\label{Fig:Denoising:Linear}
\end{figure}

\subsection{Semi-Supervised Clustering}\label{SubS:Clustering}

We consider semi-supervised clustering, where the labeling of a small set of the data is provided in advance. We introduce an optimization model based on the proposed WFTG, along with a fast numerical algorithm. The proposed model and algorithm are motivated by the earlier work on wavelet frame based image segmentation \cite{Dong2010Seg} and surface reconstruction \cite{Dong2011SurfRec}, as well as the variational frameworks for image segmentation \cite{CV,chan1999active} and graph clustering \cite{bertozzi2012diffuse,merkurjev2014global}. We shall focus on the binary clustering problem (i.e. two classes). Generalization to multiple-classes clustering can be done following a similar idea as in multiphase image segmentation \cite{sandberg2005logic,bae2011global,tai2013wavelet} which will be considered in our future studies.

Suppose the given graph $G=\{E,V,w\}$ contains two clusters. Let $|V|=K$ and $\Gamma\subset\{1,2,\ldots,K\}$ be the set of labels where, for $i\in\Gamma$, we know which cluster $v_i$ belongs to. Let $\Gamma:=\Gamma_0\cup\Gamma_1$ where $\Gamma_0$ is the index set for cluster 1 and $\Gamma_1$ for cluster 2. Define $f: \Gamma\mapsto \R$ as
\begin{equation*}
f[i]:=
\begin{cases}
0 & \mbox{for }i\in\Gamma_0,\\
1 & \mbox{for }i\in\Gamma_1.
\end{cases}
\end{equation*}
Our objective is to recover a function $u: V\mapsto [0,1]$ with $u_{|\Gamma}\approx f$, such that the two sets $\{i\ |\ u[i]\ge\beta\}$ and $\{i\ |\ u[i]<\beta\}$ provide a good clustering of the given graph for some $\beta\in (0,1)$. To find such $u$, we propose the following model
\begin{equation}\label{Model:Clustering}
\min_{u \in [0,1]}\ \|\bnu\cdot\bm{W}u\|_{1,G}+\frac12\|u_{|\Gamma}-f\|_{2,G}^2,
\end{equation}
where the first term impose a regularization of the level sets of $u$, while the second is the fidelity term making sure that $u_{|\Gamma}\approx f$. The parameter $\bnu$ is chosen as in \eqref{D:Lambda}.

To solve \eqref{Model:Clustering}, we apply the derivation of the split Bregman algorithm again, and obtain the following algorithm
\begin{equation}\label{Algorithm:Clustering2}
\begin{cases}
u^{k+1}=\arg\min_{u\in[0,1]}\ \frac12\|u_{|\Gamma}-f\|_{2,G}^2+\frac{\mu}{2}\|\bm{W}u-d^k+b^k\|_2^2,\\
d^{k+1}=\cS_{\bnu_G/\mu}(\bm{W}u^{k+1}+b^k),\\
b^{k+1}=b^{k}+\bm{W}u^{k+1}-d^{k+1},
\end{cases}
\end{equation}
where the subproblem for $u^{k+1}$ is approximated by the following steps
\begin{equation*}
\begin{cases}
\left(u^{k+\frac12}\right)_{|\Gamma^c}=\left(\bm{W}^\top(d^{k}-b^{k})\right)_{|\Gamma^c},\\
\left(u^{k+\frac12}\right)_{|\Gamma}=(D_\Gamma+\mu I)^{-1}\left(f+\mu \left(\bm{W}^\top(d^{k}-b^{k})\right)_{|\Gamma}\right),\\
u^{k+1}=\max\left\{\min\left\{u^{k+\frac12},1\right\}, 0\right\},
\end{cases}
\end{equation*}
where $D_\Gamma=\mbox{diag}\{d[k]:\ k\in\Gamma\}$.

\subsubsection{Synthetic Data}

Now, we test algorithm \eqref{Algorithm:Clustering2} using a synthetic data set known as the ``two moons'' which was first used in \cite{buhler2009spectral}. This data set is constructed from two half unit circles in $\R^2$ with centers at (0, 0) and (1, 0.5). For each circle, 1000 points are uniformly chosen and lifted to $\R^{100}$ by adding i.i.d. Gaussian white noise with variance $=0.02$ to each coordinate. The graph is formed from the point set using the exact same way as we discussed earlier in Section \ref{Sec:WFTG}.

Since the solution $u$ of \eqref{Model:Clustering} is ideally binary, it is more suitable to use the ``Haar'' tight wavelet frame system given in Example \ref{example:masks}. The level of decomposition is chosen to be $1$. Parameters are properly tuned for optimal clustering results. Algorithm \ref{Algorithm:Clustering2} is initialized by
\begin{equation}\label{Initialization}
u^0[k]:=
\begin{cases}
0, & u_1[k]>0,\\
1, & u_1[k]\le 0,
\end{cases}
\end{equation}
where $u_1$ is the second eigenvector of $\cL$, and $d^0=b^0=\bm{W}u^0$.

Same as denoising, we compare the clustering results of the model \eqref{Model:Clustering} (implemented by algorithm \eqref{Algorithm:Clustering2}) using the proposed WFTG (``Haar") and the SGWT. For SGWT, Meyer wavelet is used because numerically it produces the best clustering results among all available choices given by the codes online. Figure \ref{Fig:Clustering:2Moons} shows the ground truth of the two clusters and the clustering result by our algorithm using WFTG (``Haar") and SGWT (``Meyer") assuming 10\% of the labels of the vertices (chosen randomly) are known. More clustering results assuming different percentages of known labels are summarized in Table \ref{Table:Clustering:2Moons}. These results suggest that, for the proposed clustering model \eqref{Model:Clustering}, the WFTG is more suitable than the SGWT. In the next subsection, we shall test model \eqref{Model:Clustering} with WFTG on two realistic data sets and compare the results with some state-of-the-art clustering methods.

\begin{figure}[htp]
\centering
    \includegraphics[width=5.7in]{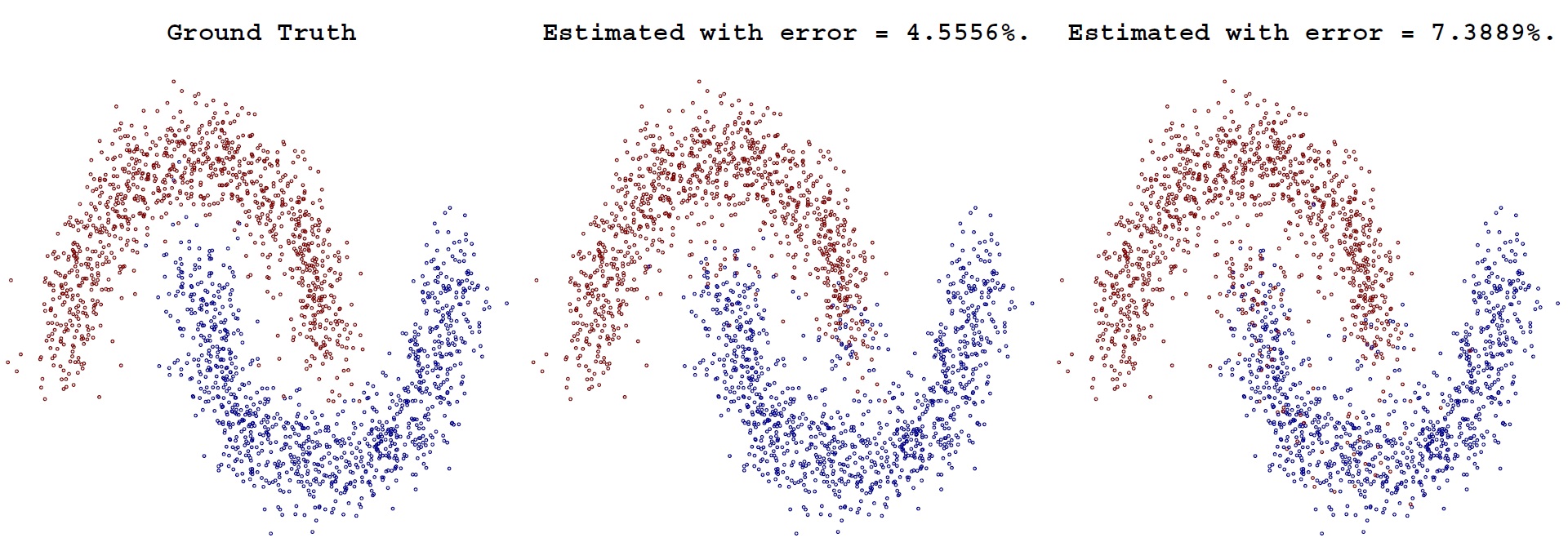}
\caption{This figure shows the ground truth of the two clusters (left) and the estimated clusters using WFTG (middle) and SGWT (right) with error rate $4.5556$\% and $7.3889$\% respectively.}\label{Fig:Clustering:2Moons}
\end{figure}

\begin{table}[htp]
\caption{{Clustering errors (\%) of the data set ``two moons'' using the WFTG and SGWT.}}\label{Table:Clustering:2Moons}
\centering
\begin{tabular}{c|c|c|c|c}
\hline
\multicolumn{1}{c|}{ Known Labels (\%) } & \multicolumn{1}{c|}{15} & \multicolumn{1}{c|}{10} &
\multicolumn{1}{c|}{5} & \multicolumn{1}{c}{3}\\
\hline
 \multicolumn{1}{c|}{WFTG}        & 4.1765 & 4.5556 &  5.9474 & 6.3402 \\
 \hline
 \multicolumn{1}{c|}{SGWT}         & 7.4118 & 7.3889 &  7.6842 & 8.7629 \\
 \hline
\end{tabular}
\end{table}

\subsubsection{Real Data Sets}

We test the proposed model \eqref{Model:Clustering} and algorithm \eqref{Algorithm:Clustering2} using the proposed WFTG on two real date sets that are often used as test sets by many recent clustering methods. We compare our method with four recently proposed semi-supervised clustering algorithms, which are: the \textbf{max-flow} method and primal augmented Lagrangian method (\textbf{PAL}) by \cite{merkurjev2014global}, the \textbf{binary MBO} scheme by \cite{merkurjev2013mbo}, and the method based on Ginzburg-Landau (\textbf{GL}) functional by \cite{bertozzi2012diffuse}. All these methods are based on solving or approximately solving a total variation (TV) based model on graphs, which take the same form as \eqref{Model:Clustering} with the first term, i.e. the regularization term, of \eqref{Model:Clustering} replaced by the graph BV (bounded variation) semi-norm of $\bu$. Comparisons with some other methods proposed earlier in the literature will also be mentioned.

The first real data set we use is the MNIST data set \cite{lecun1998mnist}, which is available at http:// yann.lecun.com/ exdb/ mnist/. It contains 70000 $28\times28$ images of handwritten digits from 0-9. Since our clustering method is binary (2-classes clustering), we choose the subset with digits 4 and 9 to classify since these digits are harder to distinguish. This created a data set of 13782 image vectors, which is either 4 or 9. In our numerical simulations, we randomly draw 500 image vectors from the data set (about 3.62\%) and use them to create the known label set $\Gamma$ and graph data $f$ of \eqref{Model:Clustering}. We repeat this process 100 times and report the average classification errors (in \%) and computation time.

The second data set is the banknote authentication data set from the UCI machine learning repository \cite{bache2013uci}. It is a data set of 1372 features extracted from images ($400 \times 400$ pixels) of genuine and forged banknotes using wavelet transforms. The goal is to separate the banknotes into being either genuine or forged. In our numerical simulations, we randomly draw 50 data vectors from the entire data set (about 3.64\%) and use them to create the known label set $\Gamma$ and graph data $f$ of \eqref{Model:Clustering}. We repeat this process 100 times and report the average classification errors (in \%) and computation time.

For both MNIST and banknote data sets, we use ``Haar" masks of Example \ref{example:masks} for the WFTG. Level of decomposition is chosen to be 1. The parameters $(\nu, \mu)$ ($\nu$ is given by \eqref{D:Lambda}) and the total number of iterations are chosen to be $(2,0.01)$ and $200$ for MNIST data set, and $(0.02,0.02)$ and $100$ for banknote data set. In all the simulations, the algorithm \eqref{Algorithm:Clustering2} is initialized using the second eigenvector of the graph Laplacian as given by \eqref{Initialization}.

Summary of the clustering results using both data sets is given in Table \ref{Table:Clustering:Real}. The results of the methods max-flow, PAL, binary MBO and GL are also included in Table \ref{Table:Clustering:Real}, which are taken from \cite{merkurjev2014global} where the exact same data sets and experimental settings were used. Our method is only slightly worse than the four latest clustering method for MNIST data set, while it is comparable to the max-flow and PAL methods and much better than binary MBO and GL for banknote data set.

There are many other clustering methods proposed in the literature that were tested on the MNIST data set as well. For example, an unsupervised clustering method was recently proposed by \cite{hu2013method}, where the classification of the digits 4 and 9 were tested. A purity score (measures the fraction of the vertices that have been assigned to the correct class) of 97.7 was reported. Also, methods based on k-nearest neighbors \cite{lecun1998gradient,lecun1998mnist}, neural or convolutional nets \cite{lecun1998gradient,lecun1998mnist,ciresan2011flexible}, and SVM \cite{decoste2002training} report clustering errors of 2.83\%-5.00\%, 0.35\%-4.70\%, and 0.68\%-1.40\%, respectively. We note that all of these approaches are supervised learning methods using 60,000 out of 70,000 digits as a training set, i.e. assuming about 85.7\% of the labels are already known. However, our method only takes 3.6\% of the entire data set as training set, and the performance is still competitive with the above approaches. Furthermore, no preprocessing or initial dimension reduction on the data set are required for our method, unlike most of the aforementioned algorithms.

\begin{table}[htp]
\caption{{Clustering errors (\%) of MNIST and banknote data set. Computation time (in seconds) for our method is included in parenthesis.}}\label{Table:Clustering:Real}
\centering
\begin{tabular}{c|c|c|c|c|c}
\hline
\multicolumn{1}{c|}{ Errors (\%) } & \multicolumn{1}{c|}{Our Method} & \multicolumn{1}{c|}{Max-Flow} &
\multicolumn{1}{c|}{PAL} & \multicolumn{1}{c|}{Binary MBO} & \multicolumn{1}{c}{GL}\\
\hline
 \multicolumn{1}{c|}{MNIST}            & 2.76 (8.5 sec.) & 1.52 &  1.56 & 1.64 & 1.75 \\
 \hline
 \multicolumn{1}{c|}{Banknote}         & 1.64 (2.9 sec.) & 1.17 &  1.71 & 6.52 & 3.90 \\
 \hline
\end{tabular}
\end{table}

\section*{Acknowledgement}

The author would like to thank the anonymous reviewers for their constructive suggestions and comments that helped tremendously in improving the presentation, clarity and rigorousness of this paper.

\bibliographystyle{ieeetr}
\bibliography{ReferenceLibrary}

\begin{thebibliography}{100}

\bibitem{hammond2011wavelets}
D.~K. Hammond, P.~Vandergheynst, and R.~Gribonval, ``Wavelets on graphs via
  spectral graph theory,'' {\em Applied and Computational Harmonic Analysis},
  vol.~30, no.~2, pp.~129--150, 2011.

\bibitem{coifman2006diffusionmaps}
R.~R. Coifman and S.~Lafon, ``Diffusion maps,'' {\em Applied and computational
  harmonic analysis}, vol.~21, no.~1, pp.~5--30, 2006.

\bibitem{belkin2005towards}
M.~Belkin and P.~Niyogi, ``Towards a theoretical foundation for laplacian-based
  manifold methods,'' in {\em Learning theory}, pp.~486--500, Springer, 2005.

\bibitem{hein2005graphs}
M.~Hein, J.-Y. Audibert, and U.~Von~Luxburg, ``From graphs to manifolds--weak
  and strong pointwise consistency of graph laplacians,'' in {\em Learning
  theory}, pp.~470--485, Springer, 2005.

\bibitem{gine2006empirical}
E.~Gin{\'e}, V.~Koltchinskii, {\em et~al.}, ``Empirical graph laplacian
  approximation of laplace--beltrami operators: Large sample results,'' in {\em
  High dimensional probability}, pp.~238--259, Institute of Mathematical
  Statistics, 2006.

\bibitem{singer2006graph}
A.~Singer, ``From graph to manifold laplacian: The convergence rate,'' {\em
  Applied and Computational Harmonic Analysis}, vol.~21, no.~1, pp.~128--134,
  2006.

\bibitem{chai2007deconvolution}
A.~Chai and Z.~Shen, ``{Deconvolution: A wavelet frame approach},'' {\em
  Numerische Mathematik}, vol.~106, no.~4, pp.~529--587, 2007.

\bibitem{chan2003wavelet}
R.~Chan, T.~Chan, L.~Shen, and Z.~Shen, ``{Wavelet algorithms for
  high-resolution image reconstruction},'' {\em SIAM Journal on Scientific
  Computing}, vol.~24, no.~4, pp.~1408--1432, 2003.

\bibitem{chan2004tight}
R.~Chan, S.~Riemenschneider, L.~Shen, and Z.~Shen, ``{Tight frame: an efficient
  way for high-resolution image reconstruction},'' {\em Applied and
  Computational Harmonic Analysis}, vol.~17, no.~1, pp.~91--115, 2004.

\bibitem{cai2008restorationchopnod}
J.~Cai, R.~Chan, L.~Shen, and Z.~Shen, ``{Restoration of chopped and nodded
  images by framelets},'' {\em SIAM J. Sci. Comput}, vol.~30, no.~3,
  pp.~1205--1227, 2008.

\bibitem{CCSS}
J.~Cai, R.~Chan, L.~Shen, and Z.~Shen, ``Convergence analysis of tight framelet
  approach for missing data recovery,'' {\em Advances in Computational
  Mathematics}, vol.~31, no.~1, pp.~87--113, 2009.

\bibitem{cai2008framelet}
J.~Cai, R.~Chan, and Z.~Shen, ``{A framelet-based image inpainting
  algorithm},'' {\em Applied and Computational Harmonic Analysis}, vol.~24,
  no.~2, pp.~131--149, 2008.

\bibitem{cai2008simultaneous}
J.~Cai, R.~Chan, and Z.~Shen, ``Simultaneous cartoon and texture inpainting,''
  {\em Inverse Problems and Imaging (IPI)}, vol.~4, no.~3, pp.~379--395, 2010.

\bibitem{cai2008frameletJCM}
J.~Cai and Z.~Shen, ``{Framelet based deconvolution},'' {\em J. Comp. Math},
  vol.~28, no.~3, pp.~289--308, 2010.

\bibitem{chan2007frameletvideo}
R.~Chan, Z.~Shen, and T.~Xia, ``{A framelet algorithm for enhancing video
  stills},'' {\em Applied and Computational Harmonic Analysis}, vol.~23, no.~2,
  pp.~153--170, 2007.

\bibitem{cai2009split}
J.~Cai, S.~Osher, and Z.~Shen, ``{Split Bregman methods and frame based image
  restoration},'' {\em Multiscale Modeling and Simulation: A SIAM
  Interdisciplinary Journal}, vol.~8, no.~2, pp.~337--369, 2009.

\bibitem{CDOS2011}
J.~Cai, B.~Dong, S.~Osher, and Z.~Shen, ``Image restorations: total variation,
  wavelet frames and beyond,'' {\em Journal of American Mathematical Society},
  vol.~25(4), pp.~1033--1089, 2012.

\bibitem{DJS2013}
B.~Dong, Q.~Jiang, and Z.~Shen, ``Image restoration: Wavelet frame shrinkage,
  nonlinear evolution pdes, and beyond,'' {\em UCLA CAM Report}, vol.~13-78,
  2013.

\bibitem{CDS2014}
J.~Cai, B.~Dong, and Z.~Shen, ``Image restorations: a wavelet frame based model
  for piecewise smooth functions and beyond,'' {\em Applied and Computational
  Harmonic Analysis}, 2015.
\newblock http://dx.doi.org/10.1016/j.acha.2015.06.009.

\bibitem{ron1997affine}
A.~Ron and Z.~Shen, ``Affine systems in {$L_2(\mathbb{R}^d)$}: The analysis of
  the analysis operator,'' {\em Journal of Functional Analysis}, vol.~148,
  no.~2, pp.~408--447, 1997.

\bibitem{Dau}
I.~Daubechies, {\em {Ten lectures on wavelets}}, vol.~CBMS-NSF Lecture Notes,
  SIAM, nr. 61.
\newblock Society for Industrial and Applied Mathematics, 1992.

\bibitem{chui1992introduction}
C.~K. Chui, {\em An introduction to wavelets}, vol.~1.
\newblock Academic press, 1992.

\bibitem{mallat2008wavelet}
S.~Mallat, {\em A wavelet tour of signal processing: the sparse way}.
\newblock Access Online via Elsevier, 2008.

\bibitem{Dong2010IASNotes}
B.~Dong and Z.~Shen, ``{MRA-Based Wavelet Frames and Applications},'' {\em IAS
  Lecture Notes Series, Summer Program on ``The Mathematics of Image
  Processing", Park City Mathematics Institute}, 2010.

\bibitem{SS}
P.~Schr{\"o}der and W.~Sweldens, ``{Spherical wavelets: Efficiently
  representing functions on the sphere},'' in {\em Proceedings of the 22nd
  annual conference on Computer graphics and interactive techniques},
  pp.~161--172, ACM New York, NY, USA, 1995.

\bibitem{Sweldens}
W.~Sweldens, ``{The lifting scheme: A construction of second generation
  wavelets},'' {\em SIAM Journal on Mathematical Analysis}, vol.~29, no.~2,
  p.~511, 1998.

\bibitem{antoine1998wavelets}
J.-P. Antoine, P.~Vandergheynst, {\em et~al.}, ``Wavelets on the n-sphere and
  related manifolds,'' {\em Journal of Mathematical Physics}, vol.~39, no.~8,
  p.~3987, 1998.

\bibitem{antoine1999wavelets}
J.-P. Antoine and P.~Vandergheynst, ``Wavelets on the 2-sphere: A
  group-theoretical approach,'' {\em Applied and Computational Harmonic
  Analysis}, vol.~7, no.~3, pp.~262--291, 1999.

\bibitem{bertram2004biorthogonal}
M.~Bertram, ``Biorthogonal loop-subdivision wavelets,'' {\em Computing},
  vol.~72, no.~1-2, pp.~29--39, 2004.

\bibitem{bertram2004generalized}
M.~Bertram, M.~A. Duchaineau, B.~Hamann, and K.~I. Joy, ``Generalized b-spline
  subdivision-surface wavelets for geometry compression,'' {\em Visualization
  and Computer Graphics, IEEE Transactions on}, vol.~10, no.~3, pp.~326--338,
  2004.

\bibitem{jiang2011biorthogonal}
Q.~Jiang, ``Biorthogonal wavelets with 4-fold axial symmetry for quadrilateral
  surface multiresolution processing,'' {\em Advances in Computational
  Mathematics}, vol.~34, no.~2, pp.~127--165, 2011.

\bibitem{jiang2011biorthogonal6}
Q.~Jiang, ``Biorthogonal wavelets with 6-fold axial symmetry for hexagonal data
  and triangle surface multiresolution processing,'' {\em International Journal
  of Wavelets, Multiresolution and Information Processing}, vol.~9, 2011.

\bibitem{wang2006efficient}
H.~Wang, K.~Qin, and K.~Tang, ``Efficient wavelet construction with
  catmull--clark subdivision,'' {\em The Visual Computer}, vol.~22, no.~9-11,
  pp.~874--884, 2006.

\bibitem{wang2007sqrt3}
H.~Wang, K.~Qin, and H.~Sun, ``$\sqrt3$-subdivision-based biorthogonal
  wavelets,'' {\em Visualization and Computer Graphics, IEEE Transactions on},
  vol.~13, no.~5, pp.~914--925, 2007.

\bibitem{khodakovsky2000progressive}
A.~Khodakovsky, P.~Schr\"{o}der, and W.~Sweldens, ``{Progressive geometry
  compression},'' in {\em Proceedings of the 27th annual conference on Computer
  graphics and interactive techniques}, pp.~271--278, Citeseer, 2000.

\bibitem{han2005wavelets}
B.~Han and Z.~Shen, ``Wavelets from the loop scheme,'' {\em Journal of Fourier
  Analysis and Applications}, vol.~11, no.~6, pp.~615--637, 2005.

\bibitem{jiang2011highly}
Q.~Jiang and D.~K. Pounds, ``Highly symmetric bi-frames for triangle surface
  multiresolution processing,'' {\em Applied and Computational Harmonic
  Analysis}, vol.~31, no.~3, pp.~370--391, 2011.

\bibitem{dong2015surf}
B.~Dong, Q.~Jiang, C.~Liu, and Z.~Shen, ``Multiscale representation of surfaces
  by tight wavelet frames with applications to denoising,'' {\em Applied and
  Computational Harmonic Analysis}, 2015.
\newblock http://dx.doi.org/10.1016/j.acha.2015.03.005.

\bibitem{crovella2003graph}
M.~Crovella and E.~Kolaczyk, ``Graph wavelets for spatial traffic analysis,''
  in {\em INFOCOM 2003. Twenty-Second Annual Joint Conference of the IEEE
  Computer and Communications. IEEE Societies}, vol.~3, pp.~1848--1857, IEEE,
  2003.

\bibitem{jansen2009multiscale}
M.~Jansen, G.~P. Nason, and B.~W. Silverman, ``Multiscale methods for data on
  graphs and irregular multidimensional situations,'' {\em Journal of the Royal
  Statistical Society: Series B (Statistical Methodology)}, vol.~71, no.~1,
  pp.~97--125, 2009.

\bibitem{murtagh2007haar}
F.~Murtagh, ``The haar wavelet transform of a dendrogram,'' {\em Journal of
  Classification}, vol.~24, no.~1, pp.~3--32, 2007.

\bibitem{lee2008treelets}
A.~B. Lee, B.~Nadler, and L.~Wasserman, ``Treelets: an adaptive multi-scale
  basis for sparse unordered data,'' {\em The Annals of Applied Statistics},
  pp.~435--471, 2008.

\bibitem{coifman2006diffusion}
R.~R. Coifman and M.~Maggioni, ``Diffusion wavelets,'' {\em Applied and
  Computational Harmonic Analysis}, vol.~21, no.~1, pp.~53--94, 2006.

\bibitem{maggioni2008diffusion}
M.~Maggioni and H.~Mhaskar, ``Diffusion polynomial frames on metric measure
  spaces,'' {\em Applied and Computational Harmonic Analysis}, vol.~24, no.~3,
  pp.~329--353, 2008.

\bibitem{geller2009continuous}
D.~Geller and A.~Mayeli, ``Continuous wavelets on compact manifolds,'' {\em
  Mathematische Zeitschrift}, vol.~262, no.~4, pp.~895--927, 2009.

\bibitem{gavish2010multiscale}
M.~Gavish, B.~Nadler, and R.~R. Coifman, ``Multiscale wavelets on trees, graphs
  and high dimensional data: Theory and applications to semi supervised
  learning,'' in {\em Proceedings of the 27th International Conference on
  Machine Learning (ICML-10)}, pp.~367--374, 2010.

\bibitem{leonardi2013tight}
N.~Leonardi and D.~Van De~Ville, ``Tight wavelet frames on multislice graphs,''
  {\em Signal Processing, IEEE Transactions on}, vol.~61, no.~13,
  pp.~3357--3367, 2013.

\bibitem{gavish2012sampling}
M.~Gavish and R.~R. Coifman, ``Sampling, denoising and compression of matrices
  by coherent matrix organization,'' {\em Applied and Computational Harmonic
  Analysis}, vol.~33, no.~3, pp.~354--369, 2012.

\bibitem{chui2014representation}
C.~Chui, F.~Filbir, and H.~Mhaskar, ``Representation of functions on big data:
  graphs and trees,'' {\em Applied and Computational Harmonic Analysis}, 2014.

\bibitem{ron1998compactly}
A.~Ron and Z.~Shen, ``{Compactly supported tight affine spline frames in
  $L_2(\mathbb{R}^d)$},'' {\em Mathematics of Computation}, vol.~67, no.~221,
  pp.~191--207, 1998.

\bibitem{grochenig1998tight}
K.~Gr{\"o}chenig and A.~Ron, ``Tight compactly supported wavelet frames of
  arbitrarily high smoothness,'' {\em Proceedings of the American Mathematical
  Society}, vol.~126, no.~4, pp.~1101--1107, 1998.

\bibitem{chui2000compactly}
C.~Chui and W.~He, ``{Compactly supported tight frames associated with
  refinable functions},'' {\em Applied and Computational Harmonic Analysis},
  vol.~8, no.~3, pp.~293--319, 2000.

\bibitem{petukhov2001explicit}
A.~Petukhov, ``Explicit construction of framelets,'' {\em Applied and
  Computational Harmonic Analysis}, vol.~11, no.~2, pp.~313--327, 2001.

\bibitem{selesnick2001smooth}
I.~Selesnick, ``{Smooth Wavelet Tight Frames with Zero Moments* 1},'' {\em
  Applied and Computational Harmonic Analysis}, vol.~10, no.~2, pp.~163--181,
  2001.

\bibitem{petukhov2003symmetric}
A.~Petukhov, ``Symmetric framelets,'' {\em Constructive Approximation},
  vol.~19, no.~2, p.~309, 2003.

\bibitem{Daubechies2003}
I.~Daubechies, B.~Han, A.~Ron, and Z.~Shen, ``{Framelets: MRA-based
  constructions of wavelet frames},'' {\em Applied and Computational Harmonic
  Analysis}, vol.~14, pp.~1--46, Jan 2003.

\bibitem{DSpseudospline}
B.~Dong and Z.~Shen, ``{Pseudo-splines, wavelets and framelets},'' {\em Applied
  and Computational Harmonic Analysis}, vol.~22, no.~1, pp.~78--104, 2007.

\bibitem{han2009dual}
B.~Han and Z.~Shen, ``{Dual wavelet frames and Riesz bases in Sobolev
  spaces},'' {\em Constructive Approximation}, vol.~29, no.~3, pp.~369--406,
  2009.

\bibitem{mason2010chebyshev}
J.~C. Mason and D.~C. Handscomb, {\em Chebyshev polynomials}.
\newblock CHAPMAN \& HALL/CRC Press, 2010.

\bibitem{chavel1984eigenvalues}
I.~Chavel, {\em Eigenvalues in Riemannian geometry}, vol.~115.
\newblock Academic press, 1984.

\bibitem{weyl1912asymptotische}
H.~Weyl, ``Das asymptotische verteilungsgesetz der eigenwerte linearer
  partieller differentialgleichungen (mit einer anwendung auf die theorie der
  hohlraumstrahlung),'' {\em Mathematische Annalen}, vol.~71, no.~4,
  pp.~441--479, 1912.

\bibitem{grieser2002uniform}
D.~Grieser, ``Uniform bounds for eigenfunctions of the laplacian on manifolds
  with boundary*,'' {\em Communications in Partial Differential Equations},
  vol.~27, no.~7-8, pp.~1283--1299, 2002.

\bibitem{cavaretta1991stationary}
A.~Cavaretta, W.~Dahmen, and C.~Micchelli, {\em {Stationary subdivision}}.
\newblock Amer Mathematical Society, 1991.

\bibitem{jiang1999existence}
Q.~Jiang and Z.~Shen, ``{On existence and weak stability of matrix refinable
  functions},'' {\em Constructive Approximation}, vol.~15, no.~3, pp.~337--353,
  1999.

\bibitem{mallat1989multiresolution}
S.~Mallat, ``{Multiresolution approximations and wavelet orthonormal bases of L
  2 (R)},'' {\em Transactions of the American Mathematical Society}, vol.~315,
  no.~1, pp.~69--87, 1989.

\bibitem{meyer1992wavelets}
Y.~Meyer, {\em {Wavelets and operators. Translated by DH Salinger}}.
\newblock Cambridge Studies in Advanced Mathematics, 1992.

\bibitem{de1993construction}
C.~De~Boor, R.~DeVore, and A.~Ron, ``{On the construction of multivariate (pre)
  wavelets},'' {\em Constructive approximation}, vol.~9, no.~2, pp.~123--166,
  1993.

\bibitem{jia1994multiresolution}
R.~Jia and Z.~Shen, ``{Multiresolution and wavelets},'' {\em Proc. Edinb. Math.
  Soc., II. Ser.}, vol.~37, no.~2, pp.~271--300, 1994.

\bibitem{daubechies1988orthonormal}
I.~Daubechies, ``{Orthonormal bases of compactly supported wavelets},'' {\em
  Commun. Pure Appl. Math.}, vol.~41, no.~7, pp.~909--996, 1988.

\bibitem{cohen1992biorthogonal}
A.~Cohen, I.~Daubechies, and J.~Feauveau, ``{Biorthogonal Bases of Compactly
  Supported Wavelets},'' {\em Communications on Pure and Applied Mathematics},
  vol.~45, no.~5, pp.~485--560, 1992.

\bibitem{chui2002compactly}
C.~Chui, W.~He, and J.~St\"{o}ckler, ``{Compactly supported tight and sibling
  frames with maximum vanishing moments},'' {\em Applied and computational
  harmonic analysis}, vol.~13, no.~3, pp.~224--262, 2002.

\bibitem{fan2014duality}
Z.~Fan, A.~Heinecke, and Z.~Shen, ``Duality for frames,'' {\em Journal of
  Fourier Analysis and Applications}, 2015, DOI: 10.1007/s00041-015-9415-0.

\bibitem{ROF}
L.~Rudin, S.~Osher, and E.~Fatemi, ``{Nonlinear total variation based noise
  removal algorithms},'' {\em Phys. D}, vol.~60, pp.~259--268, 1992.

\bibitem{Gilboa2008}
G.~Gilboa and S.~Osher, ``Nonlocal operators with applications to image
  processing,'' {\em Multiscale Model Sim}, vol.~7, pp.~1005--1028, Jan 2008.

\bibitem{GO}
G.~Gilboa and S.~Osher, ``{Nonlocal linear image regularization and supervised
  segmentation},'' {\em Multiscale Modeling and Simulation}, vol.~6, no.~2,
  pp.~595--630, 2008.

\bibitem{bertozzi2012diffuse}
A.~L. Bertozzi and A.~Flenner, ``Diffuse interface models on graphs for
  classification of high dimensional data,'' {\em Multiscale Modeling \&
  Simulation}, vol.~10, no.~3, pp.~1090--1118, 2012.

\bibitem{merkurjev2013mbo}
E.~Merkurjev, T.~Kostic, and A.~L. Bertozzi, ``An mbo scheme on graphs for
  classification and image processing,'' {\em SIAM Journal on Imaging
  Sciences}, vol.~6, no.~4, pp.~1903--1930, 2013.

\bibitem{merkurjev2014global}
E.~Merkurjev, E.~Bae, A.~Bertozzi, and X.~C. Tai, ``Global binary optimization
  on graphs for classification of high dimensional data,'' {\em UCLA CAM
  Report}, vol.~14-72, 2014.

\bibitem{mumford1989optimal}
D.~Mumford and J.~Shah, ``Optimal approximations by piecewise smooth functions
  and associated variational problems,'' {\em Communications on pure and
  applied mathematics}, vol.~42, no.~5, pp.~577--685, 1989.

\bibitem{elad2005simultaneous}
M.~Elad, J.~Starck, P.~Querre, and D.~Donoho, ``{Simultaneous cartoon and
  texture image inpainting using morphological component analysis (MCA)},''
  {\em Applied and Computational Harmonic Analysis}, vol.~19, no.~3,
  pp.~340--358, 2005.

\bibitem{starck2005image}
J.~Starck, M.~Elad, and D.~Donoho, ``{Image decomposition via the combination
  of sparse representations and a variational approach},'' {\em IEEE
  transactions on image processing}, vol.~14, no.~10, pp.~1570--1582, 2005.

\bibitem{GoldO}
T.~Goldstein and S.~Osher, ``{The split Bregman algorithm for L1 regularized
  problems},'' {\em SIAM Journal on Imaging Sciences}, vol.~2, no.~2,
  pp.~323--343, 2009.

\bibitem{gabay1976dual}
D.~Gabay and B.~Mercier, ``A dual algorithm for the solution of nonlinear
  variational problems via finite element approximation,'' {\em Computers \&
  Mathematics with Applications}, vol.~2, no.~1, pp.~17--40, 1976.

\bibitem{bertsekas1989parallel}
D.~Bertsekas and J.~Tsitsiklis, {\em Parallel and distributed computation:
  numerical methods}.
\newblock Prentice-Hall, Inc., 1989.

\bibitem{eckstein1992douglas}
J.~Eckstein and D.~Bertsekas, ``On the douglas¡ªrachford splitting method and
  the proximal point algorithm for maximal monotone operators,'' {\em
  Mathematical Programming}, vol.~55, no.~1, pp.~293--318, 1992.

\bibitem{le2007variational}
T.~Le, R.~Chartrand, and T.~J. Asaki, ``A variational approach to
  reconstructing images corrupted by poisson noise,'' {\em Journal of
  Mathematical Imaging and Vision}, vol.~27, no.~3, pp.~257--263, 2007.

\bibitem{chan2007multilevel}
R.~H. Chan and K.~Chen, ``Multilevel algorithm for a poisson noise removal
  model with total-variation regularization,'' {\em International Journal of
  Computer Mathematics}, vol.~84, no.~8, pp.~1183--1198, 2007.

\bibitem{zhang2008wavelets}
B.~Zhang, J.~M. Fadili, and J.-L. Starck, ``Wavelets, ridgelets, and curvelets
  for poisson noise removal,'' {\em Image Processing, IEEE Transactions on},
  vol.~17, no.~7, pp.~1093--1108, 2008.

\bibitem{Dong2010Seg}
B.~Dong, A.~Chien, and Z.~Shen, ``Frame based segmentation for medical
  images,'' {\em Communications in Mathematical Sciences}, vol.~9(2),
  pp.~551--559, 2010.

\bibitem{Dong2011SurfRec}
B.~Dong and Z.~Shen, ``Frame based surface reconstruction from unorganized
  points,'' {\em Journal of Computational Physics}, vol.~230, pp.~8247--8255,
  2011.

\bibitem{CV}
F.~Chan and L.~Vese, ``{Active contours without edges},'' {\em IEEE
  Transactions on image processing}, vol.~10, no.~2, pp.~266--277, 2001.

\bibitem{chan1999active}
F.~Chan and L.~Vese, ``{An active contour model without edges},'' {\em
  Scale-Space Theories in Computer Vision}, vol.~1682, pp.~141--151, 1999.

\bibitem{sandberg2005logic}
B.~Sandberg and T.~F. Chan, ``A logic framework for active contours on
  multi-channel images,'' {\em Journal of Visual Communication and Image
  Representation}, vol.~16, no.~3, pp.~333--358, 2005.

\bibitem{bae2011global}
E.~Bae, J.~Yuan, and X.-C. Tai, ``Global minimization for continuous multiphase
  partitioning problems using a dual approach,'' {\em International journal of
  computer vision}, vol.~92, no.~1, pp.~112--129, 2011.

\bibitem{tai2013wavelet}
C.~Tai, X.~Zhang, and Z.~Shen, ``Wavelet frame based multiphase image
  segmentation,'' {\em SIAM Journal on Imaging Sciences}, vol.~6, no.~4,
  pp.~2521--2546, 2013.

\bibitem{buhler2009spectral}
T.~B{\"u}hler and M.~Hein, ``Spectral clustering based on the graph
  p-laplacian,'' in {\em Proceedings of the 26th Annual International
  Conference on Machine Learning}, pp.~81--88, ACM, 2009.

\bibitem{lecun1998mnist}
Y.~LeCun and C.~Cortes, ``The mnist database of handwritten digits,'' {\em URL
  http://yann.lecun.com/exdb/mnist/}, 1998.

\bibitem{bache2013uci}
K.~Bache and M.~Lichman, ``Uci machine learning repository,'' {\em URL
  http://archive. ics. uci. edu/ml}, vol.~901, 2013.

\bibitem{hu2013method}
H.~Hu, T.~Laurent, M.~A. Porter, and A.~L. Bertozzi, ``A method based on total
  variation for network modularity optimization using the mbo scheme,'' {\em
  SIAM Journal on Applied Mathematics}, vol.~73, no.~6, pp.~2224--2246, 2013.

\bibitem{lecun1998gradient}
Y.~LeCun, L.~Bottou, Y.~Bengio, and P.~Haffner, ``Gradient-based learning
  applied to document recognition,'' {\em Proceedings of the IEEE}, vol.~86,
  no.~11, pp.~2278--2324, 1998.

\bibitem{ciresan2011flexible}
D.~C. Ciresan, U.~Meier, J.~Masci, L.~Maria~Gambardella, and J.~Schmidhuber,
  ``Flexible, high performance convolutional neural networks for image
  classification,'' in {\em IJCAI Proceedings-International Joint Conference on
  Artificial Intelligence}, vol.~22, p.~1237, 2011.

\bibitem{decoste2002training}
D.~Decoste and B.~Sch{\"o}lkopf, ``Training invariant support vector
  machines,'' {\em Machine learning}, vol.~46, no.~1-3, pp.~161--190, 2002.

\end{thebibliography}

\end{document}